\documentclass[a4paper,10pt,reqno,twoside]{amsart}

\usepackage{mathtools,amsmath,amssymb,amsfonts,amsthm,mathrsfs}
\usepackage[pdftex]{graphicx}
\usepackage{graphicx}
\mathtoolsset{showonlyrefs}

\DeclareGraphicsExtensions{.bmp,.png,.pdf,.jpg,}

\usepackage{bm}
\usepackage{verbatim}
\usepackage{color}
\usepackage{mathrsfs}
\usepackage{pgfplots}
\usepackage{euscript}
\usepackage{bbm}
\usepackage{cite}

\newcommand{\limess}[1]%
{

\begin{array}[t]{c}
{\rm ess\, lim}\\
{\scriptstyle #1}
\end{array}

}

\newcommand{\supess}[1]%
{

\begin{array}[t]{c}
{\rm ess\, sup}\\
{\scriptstyle #1}
\end{array}

}
\newcommand{\infess}[1]%
{

\begin{array}[t]{c}
{\rm ess\, inf}\\
{\scriptstyle #1}
\end{array}

}

\def\R{{\mathbb R}}

\def\d{\delta}
\def\D{\Delta}

\def\s{\sigma}

\def\l{\lambda}
\def\p{\partial}
\def\O{\Omega}
\def\e{\varepsilon}

\def\v{\varphi}
\def\mc{\mathcal}

\def\ms{\mathscr}
\def\ua{\uparrow}
\def\da{\downarrow}

\numberwithin{equation}{section}

\theoremstyle{definition}

\theoremstyle{definition}

\theoremstyle{plain}
\newtheorem{theorem}{Theorem}[section]

\newtheorem{lemma}{Lemma}[section]

\newtheorem{definition}{Definition}[section]

\begin{document}

\title[Limiting behavior of principal eigenvalues and eigenfunctions]
{Limiting behavior of principal eigenvalues and eigenfunctions for a class of elliptic operators with
degenerate large advection}

\thanks{The authors have been  supported by the Research Grants PID2021-123343NB-I00 and PID2024-155890NB-I00 
of the
Ministry of Science and Innovation of Spain and the Institute of Interdisciplinary Mathematics of Complutense University of Madrid.}

\author{S. Cano-Casanova}
\address{S. Cano-Casanova: ICAI, Comillas Pontifical University, Madrid, Spain}
\email{scano@icai.comillas.edu}
\author{J. L\'{o}pez-G\'{o}mez}
\address{J. L\'{o}pez-G\'{o}mez:
Universidad Complutense de Madrid,
Madrid, Spain}
\email{jlopezgo@ucm.es}

\author{M. Molina-Meyer}
\address{M. Molina-Meyer:
Universidad Carlos III de Madrid,
Madrid, Spain}
\email{mmolinam@math.uc3m.es}

\begin{abstract}
\vskip.5cm
In this paper we study, both numerically and analytically,  the asymptotic behavior of the principal eigenfunction of \eqref{1.1}, normalized by \eqref{1.2}, as $s\ua +\infty$. Based on the numerical computations of this paper, we can prove that, under condition (Hm) bellow, $\v_s$ approximates $1$ and  $\v_s'$ approximates $0$, uniformly in $[-1,1]$, as $s\ua +\infty$. As a byproduct of this result, we can derive the asymptotic behavior of the principal eigenvalue in a one-dimensional situation  not previously covered by \cite{ChLo}  and \cite{PeZh}, as we are working under minimal regularity assumptions on $m(x)$. A recent result of \cite{BWZ}
shows that the principal eigenvalue might oscillate as $s\ua +\infty$ if $m(x)$ is highly oscillatory. Thus, the oscillatory and regularity properties of $m(x)$ might severely affect the asymptotic behavior of $(\l_s,\v_s)$
as $s\ua +\infty$.
\vskip.5cm
\end{abstract}


\vspace{0.1cm}

\date{\today}

\maketitle

\setcounter{page}{1}
\section{Introduction}

\noindent In this paper we analyze the limiting  behavior of the principal eigenfunction, $\v_s\equiv \v(s)$, and the principal eigenvalue, $\l_s\equiv \l(s)$, of the linear eigenvalue problem
\begin{equation}
\label{1.1}
\left\{
\begin{array}{ll}
-\v_s''-2sm'(x)\v_s' +c(x)\v_s=\lambda(s) \v_s & \quad \hbox{in}\;\; [-1,1], \\[1ex]
\v_s'(-1)=\v_s'(1)=0,&
\end{array}
\right.
\end{equation}
where $m$ satisfies
\begin{itemize}
\item[(Hm)] $m \in\mathcal{C}^1([-1,1];\R)$, and there is a unique  $x_0\in (-1,1)$ such that
\begin{equation*}
   m(x_0)=\|m\|_\infty = \max_{|x|\leq 1}m(x).
\end{equation*}
Moreover, $m'(x)\geq 0$ for all $x\in [-1,x_0)$, $m'(x)\leq 0$ for all $x\in (x_0,1]$, and $m(x)$ has, at most, finitely many critical points in $(-1,1)$.
\end{itemize}
In \eqref{1.1},  $c\in \mc{C}([-1,1];\R)$ is arbitrary,
though, eventually, we will need to focus attention on some special types of $c$'s. Throughout this paper, the principal eigenfunction is
normalized so that
\begin{equation}
\label{1.2}
\| \v_s \|_\infty =\max_{|x|\leq 1} \v_s(x) =1.
\end{equation}
In a multidimensional context, this problem was addressed by Chen and Lou in \cite{ChLo}, where, in the setting of this paper, it was established that,  when, in addition,  $m''(x_0)<0$ and the unique critical point of $m(x)$ in $[-1,1]$ is $x_0$, then
\begin{equation}
\label{1.3}
  \lim_{s\uparrow \infty}\l(s)=c(x_0).
\end{equation}
Although the theorem of Chen and Lou \cite{ChLo} cannot be applied to cover \eqref{1.1} with  $m(x)$ satisfying (Hm), according to the one-dimensional results of Peng and Zhou \cite{PeZh},  \eqref{1.3} holds under condition (Hm) if, in addition, $m\in\mc{C}^2([-1,1];\R)$. Thus, \eqref{1.3} cannot be inferred from \cite{PeZh} even in the simplest case when (Hm) holds, as (Hm) is imposing minimal regularity assumptions on $m(x)$.
\par
The analysis of the asymptotic behavior of principal eigenvalues when either the diffusion coefficient decays to zero, or the amplitud of the potential $c(x)$ grows to infinity,  had been previously analyzed by
Furter and L\'{o}pez-G\'{o}mez \cite{FLG}, Dancer and L\'{o}pez-G\'{o}mez \cite{DLG},
Cano-Casanova and L\'{o}pez-G\'{o}mez \cite{CCLG}, and L\'{o}pez-G\'{o}mez \cite{LG94,LG96}. But these references were left outside the bibliographies  of \cite{ChLo} and \cite{PeZh}.
According to Lemma 3.1 of \cite{FLG}, the principal eigenvalue of the problem
$$
  \left\{ \begin{array}{ll} -d \D \v_d+ c(x) \v_d =\l(d) \v_d & \quad \hbox{in}\;\; \O, \\
  \v_d =0    & \quad \hbox{on}\;\; \p \O,\end{array}\right.
$$
where $\O$ is a sufficiently smooth subdomain of $\R^N$, $N\geq 1$, satisfies
$$
  \lim_{d\da 0}\l(d)=\min_{\bar\O} c.
$$
Some time later, Dancer and L\'{o}pez-G\'{o}mez \cite{DLG} found out the precise asymptotic expansion of
$\v_d$ and $\l(d)$ as $d\da 0$, for a general class of Shr\"{o}dinger operators including transport terms, sharpening some previous results in quantum mechanics by Mart\'{\i}nez and Rouleux  \cite{MaRo}, Helffer \cite{Hel}, Helffer and  Sj\"{o}strand \cite{HeSj} and Simon \cite{Simon}.
Similarly, thanks to Theorem 3.3 of \cite{LG94}, if $c\geq 0$ and $c^{-1}(0)=\bar\O_0$ for some
smooth subdomain $\O_0$ of $\O$, then, the principal eigenvalue, $\l(s)$, of the problem
$$
  \left\{ \begin{array}{ll} -\D \v_d+ s c(x) \v_d =\l(s) \v_d & \quad \hbox{in}\;\; \O, \\
  \v_d =0    & \quad \hbox{on}\;\; \p \O,\end{array}\right.
$$
satisfies
$$
  \lim_{s\ua \infty}\l(s) =\s[-\D,\O_0],
$$
where $\s[-\D,\O_0]$ stands for the principal eigenvalue of $-\D$ in $\O_0$ under Dirichlet boundary
conditions on $\p\O_0$. This result was sharpened, very substantially, by L\'{o}pez-G\'{o}mez \cite{LG96},
Cano-Casanova and L\'{o}pez-G\'{o}mez  \cite{CCLG} to cover the case of general second
order elliptic operators under non-classical mixed boundary conditions  (see \cite{LG13} and the list of references therein). Some applications of these singular perturbation results to the theory of competing species were delivered by L\'{o}pez-G\'{o}mez in \cite{LG94,LG95}, Cano-Casanova and L\'{o}pez-G\'{o}mez in \cite{CCLG03}, L\'{o}pez-G\'{o}mez and Molina-Meyer in \cite{LGMMa,LGMMb}, as well as in \cite{LG15}.
\par
The main goal of this paper is to analyze the behavior of the principal eigenfunction,
$\v_s$, both numerically and analytically, as $s\ua +\infty$. As a consequence of this analysis,
we can deliver a short self-contained proof of \eqref{1.3} for some special classes of $c$'s under condition (Hm). This (elementary) proof differs, very substantially, from the highly technical and lengthly proofs of \cite{ChLo} and \cite{PeZh}, based on the variational formulation of the problem associated to the re-scaled function
$$
  w_s(x):= e^{sm(x)}\psi_s(x),\qquad |x|\leq 1,
$$
where $\psi_s$ is a positive solution of \eqref{1.1}. As in \cite{ChLo} and \cite{PeZh} the function $w_s$ is  normalized to satisfy
$$
  \int_{-1}^1 w_s^2(x)\,dx =1,
$$
necessarily,
$$
   \int_{-1}^1 e^{2sm(x)}\psi_s^2(x)\,dx =1\;\;\hbox{for all}\;\; s>0.
$$
Therefore, since one can assume that $m(x)>0$ for all $x\in [-1,1]$, because the setting of \eqref{1.1} only involves $m'\in\mc{C}([-1,1];\R)$, the eigenfunctions  $\psi_s(x)$ should approximate zero as $s\ua+\infty$, which is incompatible with \eqref{1.2}. Such a crucial difference between $\v_s$ and $\psi_s$ makes the analysis of this paper complementary with respect to that of \cite{ChLo} and \cite{PeZh}.
\par
Subsequently, for every integer $n=2\nu \geq 2$, we considered the sequence of functions
\begin{equation}
\label{1.4}
  m(x)=1-x^n,\quad |x|\leq 1,
\end{equation}
which satisfy (Hm). For this choice, the degree of degeneracy at $x_0=0$ can be arbitrarily large, and, actually,
$$
  \lim_{n\ua+\infty}(1-x^n)= 1 \quad \;\; \hbox{if}\;\; |x|<1.
$$
Figure \ref{Fig1} shows the corresponding eigenfunctions, $\v_s$, normalized by \eqref{1.2}, as well as their derivatives, $\v_s'$, and the plot of $\l(s)-c(0)$, for a series of values of $s$ ranging from $s=10^{-2}$ up to $s=10^4$, for the choice
\begin{equation}
\label{1.5}
   c(x) = 2 + x,\quad m(x) = 1 - x^4,\qquad |x|\leq 1,
\end{equation}
i.e. $n=4$ in \eqref{1.4}. As the remaining eigenfunctions of this section, they have have been calculated through a pseudospectral method that will be discussed in Section 5.

\begin{figure}[h!]
	\centering
	\includegraphics[scale=0.21]{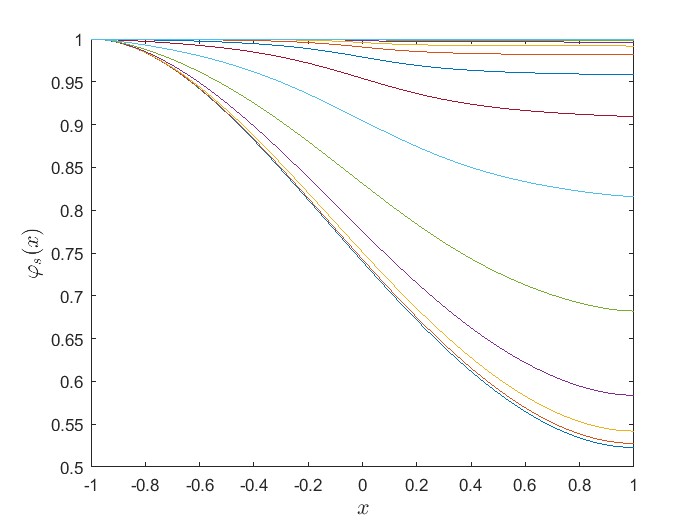}
	\includegraphics[scale=0.21]{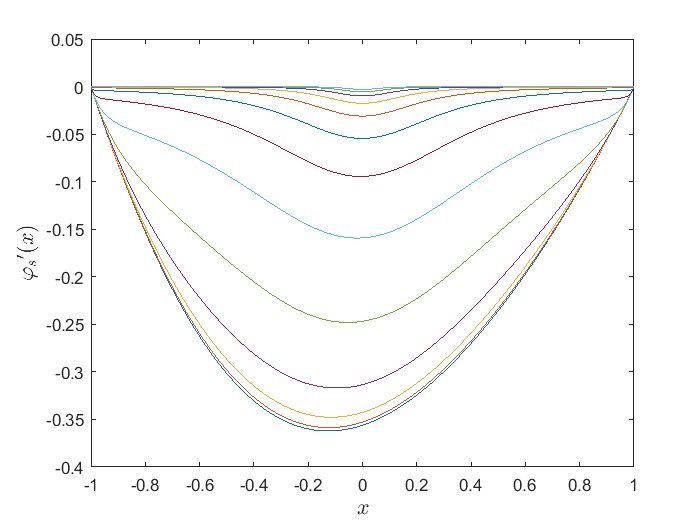}
	\includegraphics[scale=0.21]{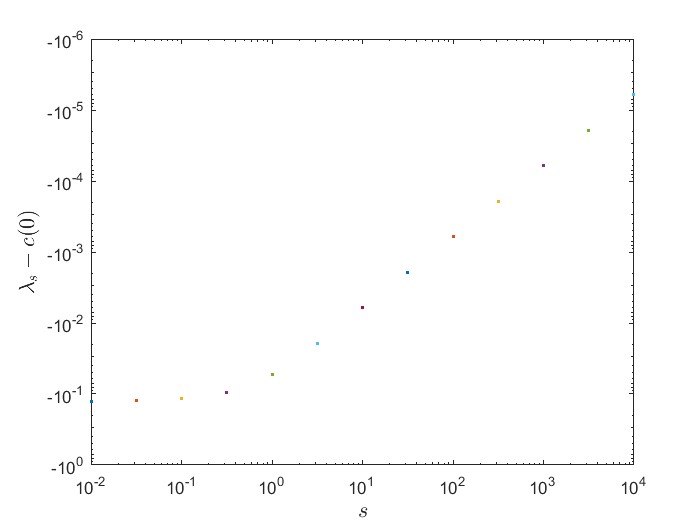}
	\caption{Plots of the normalized eigenfunctions $\v_s$ (left), their derivatives $\v_s'$ (center), and the graph of $\l_s-c(0)$, with $\l_s=\l(s)$,  (right) for a series of values of $s\in [10^{-2},10^4]$, with the  choice \eqref{1.5}.}
	\label{Fig1}
\end{figure}
\par
 As illustrated by the numerical simulations, the eigenfunctions are decreasing for all
value of $s$ for which we have computed them, and stabilize to $1$ as $s\ua +\infty$, while their derivatives
approximate $0$ as $s$ grows. Moreover,
$$
  \lim_{s\ua \infty}\l(s)=2=c(0).
$$
Thus, \eqref{1.3} holds with $x_0=0$, however
$$
   m''(0)=m'''(0)=0.
$$
Figure \ref{Fig2} shows the computed plots of $\v_s$, $\v_s'$, and the corresponding values of $\l(s)-c(0)$ for the choice
\begin{equation}
\label{1.6}
c(x) = 40\left(2-\cos(x-\tfrac{1}{3})\right),\quad  m(x) = 1 - x^4,\quad |x|\leq 1.
\end{equation}
Although for this choice, for an initial range of $s,$ $\v'_s$ changes sign and so $\v_s$ increases in the left side and decreases in  the right part of the interval,  for sufficiently large $s$, $\v_s$ and $\v'_s$ behave much like the corresponding $\v_s$ and $\v'_s$ for the choice \eqref{1.5}, the behavior of $\l(s)$ is rather different, as, according to Figure
\ref{Fig1}, it is apparent that $\l(s)$ is increasing with respect to $s$ for sufficiently large $s$, while, for the choice
\eqref{1.6}, $\l(s)$ decays as $s\ua \infty$ to the value
$c(0)=42.2017$, as expected, though $\l(s)$ is far from being always decreasing, as illustrated
 by the right plot of Figure \ref{Fig2}.

\begin{figure}[h!]
	\centering
	\includegraphics[scale=0.21]{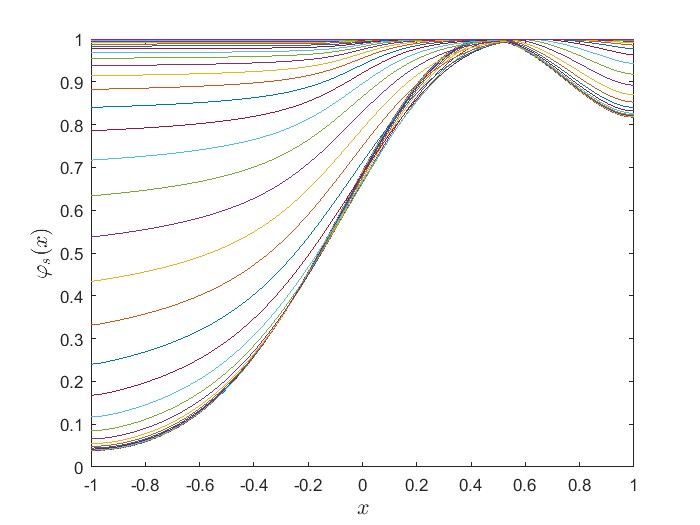}
	\includegraphics[scale=0.21]{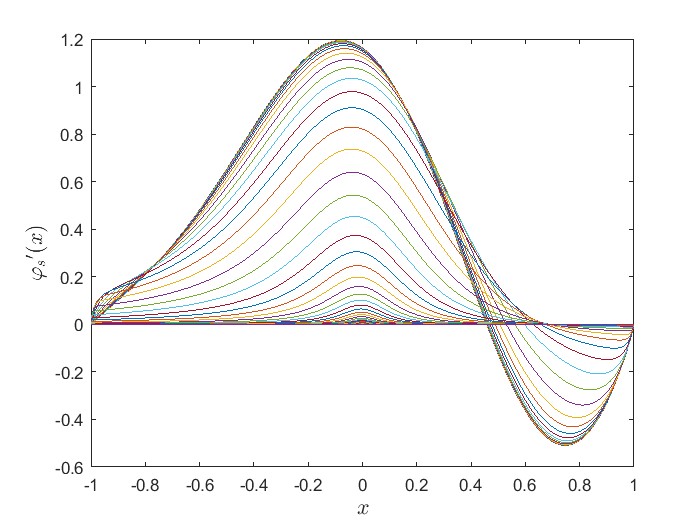}
	\includegraphics[scale=0.21]{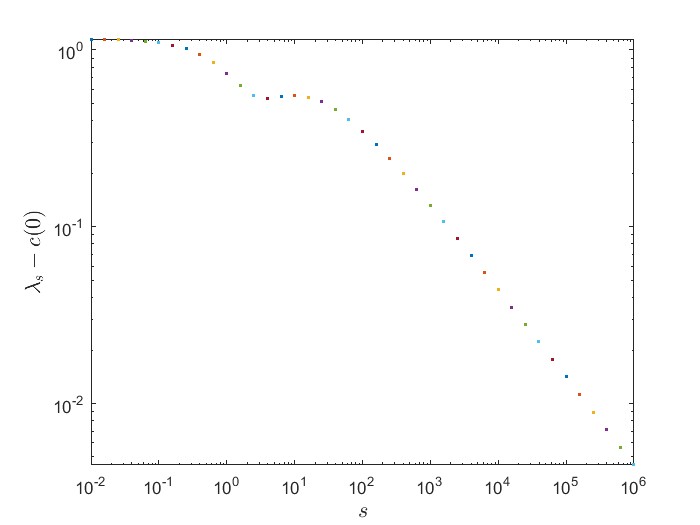}
	\caption{Plots of the normalized eigenfunctions $\v_s$ (left), their derivatives $\v_s'$ (center), and the graph of $\l(s)-c(0)$, with $\l_s=\l(s)$, (right) for a series of values of $s\in [10^{-2},10^6]$, with the  choice \eqref{1.6}.}
	\label{Fig2}
\end{figure}

Figure \ref{Fig3} shows the results of  the  numerical experiments for the choice
\begin{equation}
\label{1.7}
 c(x) = 2 - x,\quad m(x) = 1 - x^8,\quad |x|\leq 1.
\end{equation}

\begin{figure}[h!]
	\centering
	\includegraphics[scale=0.21]{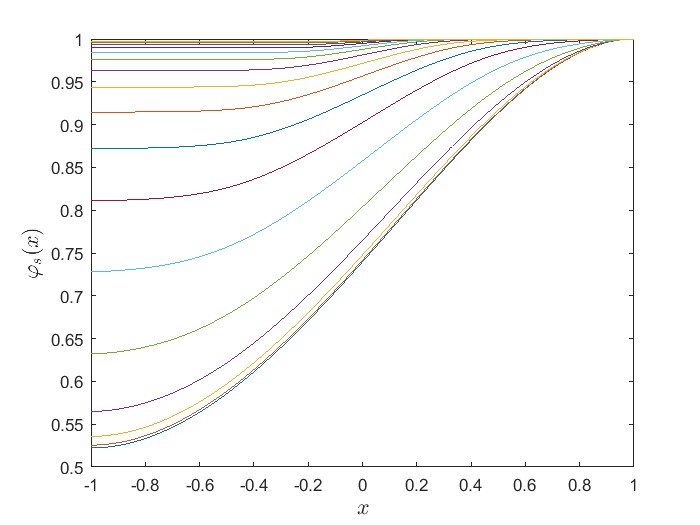}
	\includegraphics[scale=0.21]{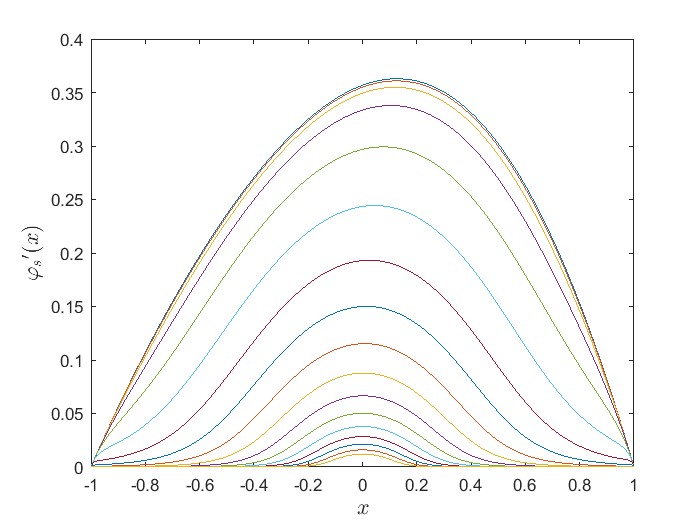}
	\includegraphics[scale=0.21]{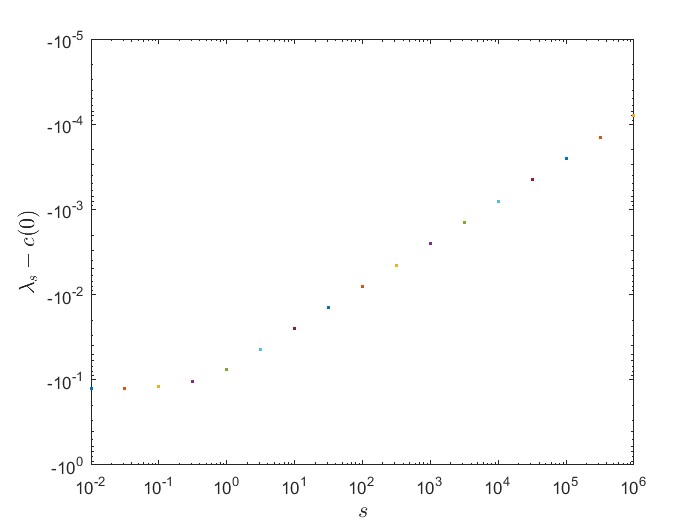}
	\caption{Plots of the normalized eigenfunctions $\v_s$ (left), their derivatives $\v_s'$ (center), and the graph of $\l(s)-c(0)$, with $\l_s=\l(s)$, (right) for a series of values of $s\in [0,10^6]$, with the choice \eqref{1.7}.}
	\label{Fig3}
\end{figure}

\noindent This case is far more singular  than \eqref{1.5} and \eqref{1.6}, as now $n=8$ in \eqref{1.4}.
In this example, $\v_s$ is increasing with respect to $x$ and, as expected,
$$
  \lim_{s\ua \infty}\l(s)=c(0)=2.
$$
However, the convergence is much slower than for the choice \eqref{1.5}, in the sense that $s$
should be much higher than before so that $\l(s)$ can approximate its limiting value, $c(0)=2$.
\par
In order to correctly compare the behavior of $\l(s)$ for different choices of $c(x)$ and $m(x)$, the number of collocation points in the numerical simulations is fixed at $N+1=802$ for all simulations of this paper (see Section 5). From there on, $\l(s)$ is calculated for a range of values of $s$ in which the numerical calculation of $\l(s)$ is stable.  Concretely, $s\in [10^{-2}, 10^4]$ or $s\in [10^{-2}, 10^6]$, as appropriate. Figure \ref{Fig4} shows the results of the numerical computations for the choice
\begin{equation}
\label{1.8}
 c(x) = 40\left(2 -\cos(x -\tfrac{1}{3})\right), \quad m(x) = 1 - x^{16},\qquad |x|\leq 1.
\end{equation}

\begin{figure}[h!]
	\centering
	\includegraphics[scale=0.21]{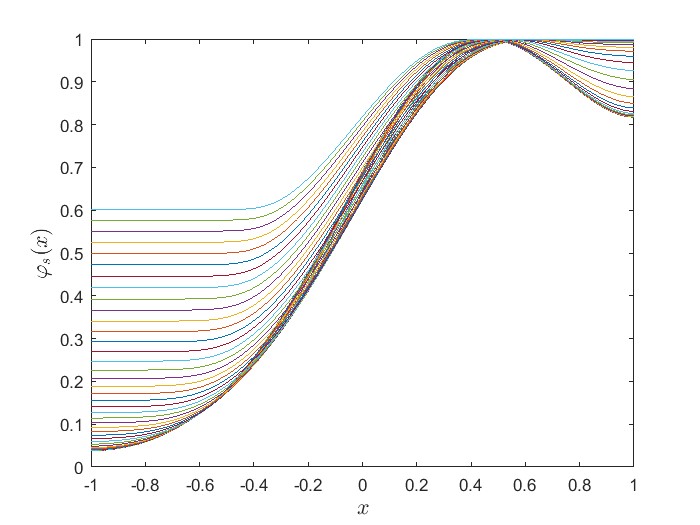}
	\includegraphics[scale=0.21]{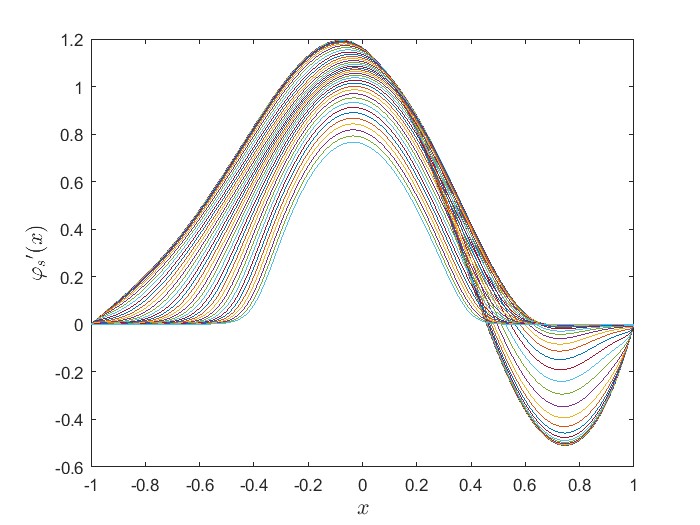}
	\includegraphics[scale=0.21]{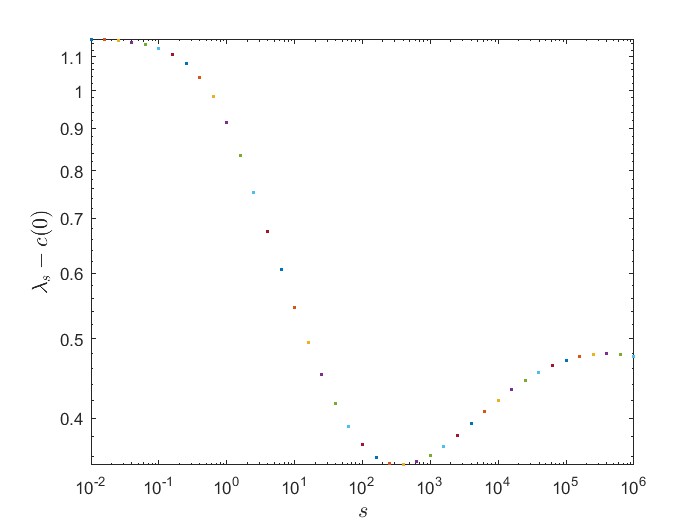}
	\caption{Plots of the normalized eigenfunctions $\v_s$ (left), their derivatives $\v_s'$ (center), and the graph of $\l(s)-c(0)$, with $\l_s=\l(s)$, (right) for a series of values of $s\in [10^{-2},10^6]$, with the choice \eqref{1.8}.}
	\label{Fig4}
\end{figure}
For the choice \eqref{1.8}, the convergence of $\l(s)$ to $c(0)$ slows down because  the value of $s$ must be incremented to catch the desired slope of $-2sm'(x)$ at 0. As the numerical computations begin to be unstable for the agreed value of $N+1=802,$ and $s=10^6$, we stopped the computations before $\v_s$ could approximate $1$ everywhere and the tangent of $\l(s)$ coincide with $c(0)= 42.2017$. Nevertheless, the reader should be aware that
$$
(0.5)^{16}\sim 0.00001525878\sim 10^{-5}.
$$
Thus,
$$
\max_{|x|\leq 0.5} |1-m(x)| \sim 0.00001525878.
$$
This quantifies very well the high degree of degeneration of $1-x^{16}$ at $x_0=0$, and illustrates the high accuracy of the numerical method used to get the plots of Figures \ref{Fig1}--\ref{Fig4}.
\par
From the numerical experiments of this paper, one can naturally infer, at least, for the choice \eqref{1.4}, the following general properties for $\v_s$, $\v_s'$ and $\l(s)$ as $s\ua \infty$:
\begin{enumerate}
\item[(P1)] $\lim_{s\ua \infty}\v_s=1$ uniformly in $[-1,1]$,
\item[(P2)] $\lim_{s\ua \infty}\v'_s=0$ uniformly in $[-1,1]$,
\item[(P3)] $\lim_{s\ua \infty}\l(s)=c(0)$.
\end{enumerate}
And this, regardless the value of the even integer $n\geq 2$. Indeed,  the next result holds (see Theorem \ref{th3.1}). It is the main theorem of this paper.

\begin{theorem}
\label{th1.1}
Suppose $c\in \mc{C}([-1,1];\R)$ and $m$ satisfies {\rm (Hm)}. Then,
\begin{equation}
\lim_{s \rightarrow \infty}\|\v_s'\|_\infty =0\quad \hbox{and}\quad \lim_{s \rightarrow \infty}
\|\v_s-1\|_\infty =0.
\end{equation}
\end{theorem}

Based on this result, the next theorem can be established.

\begin{theorem}
\label{th1.2}
Suppose $m(x)$ satisfies {\rm (Hm)}, and
\begin{enumerate}
\item[{\rm (Hc)}] the function $c(x)$ satisfies some of the following conditions:
\begin{enumerate}
\item[{\rm (a)}] there exists $y_0\in (-1,1)$ such that $c$ is increasing in $[-1,y_0)$ and decreasing in $(y_0,1]$,
\item[{\rm (b)}] there exists $y_0\in (-1,1)$ such that $c$ is decreasing in $[-1,y_0)$ and increasing in $(y_0,1]$,
\item[{\rm (c)}] $c$ is increasing in $[-1,1]$,
\item[{\rm (d)}] $c$ is decreasing in $[-1,1]$.
\end{enumerate}
\end{enumerate}
Then,
\begin{equation}
\label{1.9}
    \lim_{s \uparrow \infty}\l(s)=c(x_0).
\end{equation}
\end{theorem}

According to Theorem \ref{1.2}, it is apparent that
the condition that all critical points of $m(x)$ are non-degenerate in Theorem 1.2 of Chen and Lou \cite{ChLo} is  superfluous for the validity of \eqref{1.9}, at least, in the simplest one-dimensional examples,
as already shown by Peng and Zhou \cite{PeZh} under the
stronger assumption that $m\in\mc{C}^2([-1,1];\R)$.
\par
Going back to the differential equation of \eqref{1.1}, since $m'(x_0)=0$, we find that
$$
   -\v_s''(x_0) = -\v_s''(x_0)-2sm'(x_0)\v_s'(x_0)= (\l(s)-c(x_0))\v_s(x_0).
$$
Thus, by Theorem \ref{th1.1},  \eqref{1.9} holds true if, and only if,
\begin{equation}
\label{1.10}
  \lim_{s\ua \infty}\v_s''(x_0)=0.
\end{equation}
Although, based on Theorem \ref{th1.1}, \eqref{1.10} should be true under assumption
(Hm) for all $c\in \mc{C}([-1,1];\R)$, we were not able to get a rigorous proof of this fact, except when $c(x)$ satisfies (Hc) in Theorem \ref{th1.2}.
\par
Suppose that $m(x)$ has some critical point $x_c\in (-1,1)\setminus\{x_0\}$ and $c(x)$ satisfies (Hc). Then,  by Theorem \ref{th1.2},
$$
  \lim_{s\ua \infty}\l(s)= c(x_0).
$$
Suppose, in addition, that $c(x)$ is chosen so that $c(x_c)\neq c(x_0)$. Then, since $m'(x_c)=0$,
$$
  -\v_s''(x_c)=(\l(s)-c(x_c))\v_s(x_c) \quad \hbox{for all}\;\; s\in\R.
$$
Therefore, according to Theorems \ref{th1.1} and \ref{th1.2},
$$
   \lim_{s\to \infty}\v_s''(x_c)=c(x_c)-c(x_0)\neq 0.
$$
Therefore, in the general setting of this paper,
$$
  \lim_{s\ua \infty}\v_s''\neq 0,
$$
though for the special choices \eqref{1.4}, where $x_0=0$ is the unique critical point of $m(x)$,
our numerical simulations suggest that $\v_s''$ should approximate $0$ as $s\ua \infty$.
The analysis of the asymptotic behavior of $\v_s''$ remains as an open problem in this paper.
Figure  \ref{Fig5} collects the plots of the functions $\v_s''$ computed for the choices
\eqref{1.5}, \eqref{1.6} and \eqref{1.7}, respectively.

\begin{figure}[h!]
	\centering
	\includegraphics[scale=0.21]{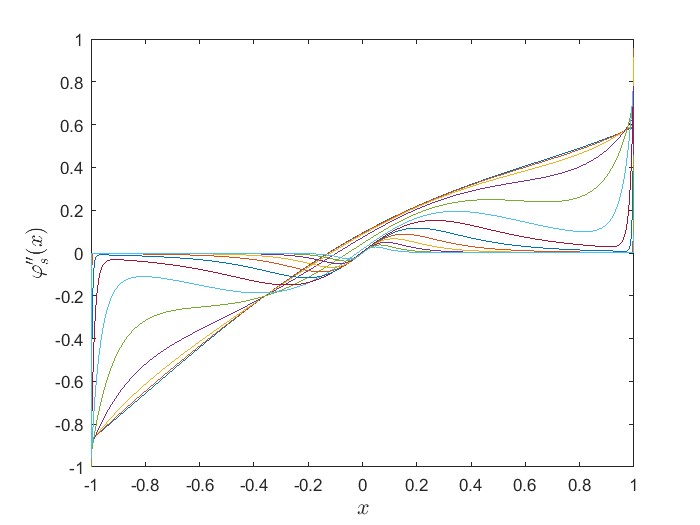}
	\includegraphics[scale=0.21]{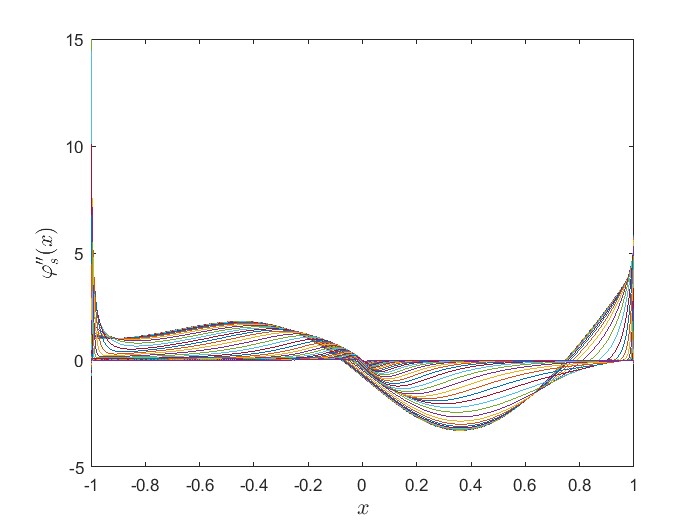}
	\includegraphics[scale=0.21]{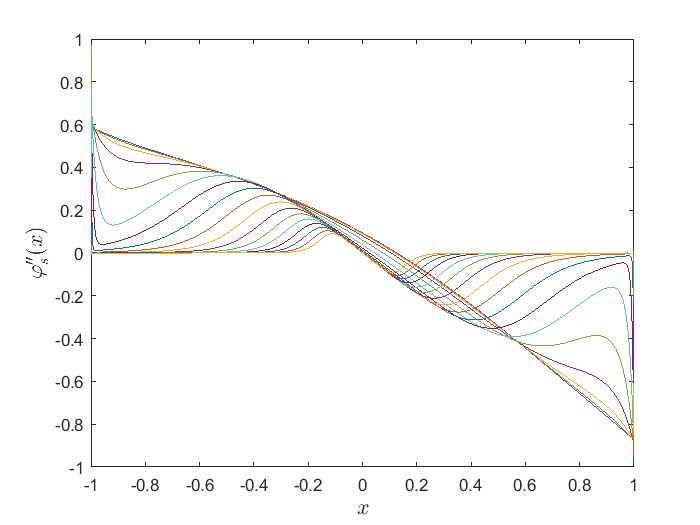}
	\caption{Plots of the computed functions $\v_s''$ for each of the choices
\eqref{1.5} (left), \eqref{1.6} (center) and \eqref{1.7} (right). }
	\label{Fig5}
\end{figure}

The distribution of this paper is the following. Section 2 collects a number of oscillatory properties of
the principal eigenfunctions $\v_s$ and their associated derivatives, $\v_s'$, that are going to be used
in the proofs of Theorems \ref{th1.1} and \ref{th1.2}. Some of them have a great interest on their own. Sections 3 and 4 deliver the proofs of Theorems \ref{th1.1} and \ref{th1.2}, respectively. Finally,
in Section 5 eight eigenvalue problems satisfying the hypotheses of this paper are numerically simulated and analyzed. These numerical examples reinforce and illustrate the analytical results.

\section{General properties of $(\l(s),\v_s)$}

\noindent Throughout this paper, we set
$$
   \ms{L}_s:=-\frac{d^2}{dx^2}-2sm'(x)\frac{d}{dx}
$$
for all $s\in\R$, and denote by $\ms{N}$ the Neumann operator on the ends of $[-1,1]$, i.e.,
$$
\ms{N}(\v):=\left(-\v'(-1),\v'(1)\right).
$$
Moreover, for every continuous functions $h\in\mc{C}([-1,1];\R)$, we denote
$$
  h_L:=\min_{|x|\leq 1}h(x), \qquad h_M:=\max_{|x|\leq 1}h(x).
$$
Using these notations, the problem \eqref{1.1} can be equivalently expressed as
\begin{equation}
\label{2.1}
\left\{
\begin{array}{l}
\ms{L}_s\v_s+c(x)\v_s=\lambda(s) \v_s  \quad \hbox{in}\;\; [-1,1], \\[1ex]
\ms{N}(\v_s) =0. \end{array} \right.
\end{equation}
The existence of the principal eigenvalue,
\begin{equation}
\label{2.2}
\lambda(s)=\sigma[\ms{L}_s+ c(x),\ms{N}],
\end{equation}
and the principal eigenfunction, $\v_s$, as well as their main properties,  can be found in
\cite[Ch. 7]{LG13}. The main aim of this section is analyzing the shape of $\v_s$, normalized so that
\eqref{1.2} holds, i.e.
$$
  (\v_s)_M=\|\v_s\|_{\infty}=1.
$$
Since $\varphi=1>0$ is a positive solution of
\begin{equation*}
\left\{
\begin{array}{l}
\ms{L}_s \varphi=0 \;\;  \hbox{in}\;\;  [-1,1], \\
\varphi'(-1)=\varphi'(1)=0, \end{array} \right.
\end{equation*}
by the uniqueness of the principal eigenvalue, it becomes apparent that
\begin{equation}
\label{2.3}
  \sigma[\ms{L}_s,\ms{N}]=0.
\end{equation}
Thus, in the special case when $c(x)$ is a constant, $c$, we have that
$$
   \lambda(s)=\sigma[\ms{L}_s+ c,\ms{N}]=c\quad \hbox{for all}\;\; s\in\R,
$$
regardless the nature of $m(x)$.
\par
The next result estimates the value of $\l(s)$ in the more general case when
$c(x)$ is non-constant. It has an obvious multidimensional counterpart, with exactly the same proof.

\begin{lemma}
\label{le2.1}
Suppose that $c\in\mc{C}[-1,1]$ is non-constant. Then, for every $s \in \R$, $\lambda (s) \in (c_L,c_M)$. Thus, there exists $x_s\in [-1,1]$ such that $\lambda(s)=c(x_s)$.
\end{lemma}
\begin{proof}
Fix $s\in\R$. By the monotonicity of the principal eigenvalue with respect to the potencial (see \cite[Sect. 3]{CCLG}), it follows from \eqref{2.2} and \eqref{2.3} that
\begin{equation*}
c_L=\sigma_1[\ms{L}_s,\ms{N}]+c_L < \lambda(s) < \sigma_1[\ms{L}_s,\ms{N}]+c_M =c_M.
\end{equation*}
By continuity of $c(x)$, the existence of $x_s$ holds. This ends the proof.
\end{proof}

The next result provides us with the values of the derivative, $\v_s'(x)$, from any critical point, $y_0\in[-1,1]$, of the eigenfunction $\v_s(x)$.

\begin{lemma}
\label{le2.2}
Suppose $y_0\in [-1,1]$ is a critical point of $\v_s$, i.e. $\v_s'(y_0)=0$. Then,
for every $x \in [-1,1]$,
\begin{equation}
\label{ii.3}
\v_s'(x)=\int_{y_0}^x(c(t)-\l(s))\v_s(t)e^{2s(m(t)-m(x))}\,dt.
\end{equation}
In particular, since $\v_s'(-1)=\v_s'(1)=0$,
\begin{equation*}
\int_{-1}^{y_0} (c(t)-\l(s))\v_s(t) e^{2sm(t)}\,dt=0=\int_{y_0}^1 (c(t)-\l(s))\v_s(t) e^{2sm(t)}\,dt.
\end{equation*}
\end{lemma}
\vspace{0.2cm}

Note that, making the choices $y_0=-1$ and $y_0=1$ in \eqref{ii.3}, we find that
\begin{equation}
\label{ii.4}
\begin{split}
\v_s'(x) & =\int_{-1}^x(c(t)-\l(s))\v_s(t)e^{2s(m(t)-m(x))}\,dt\\
      & = \int_x^1 (\l(s)-c(t))\v_s(t)e^{2s(m(t)-m(x))}\,dt
\end{split}
\end{equation}
for all $x\in [-1,1]$.

\begin{proof}
Making the change of variable $u_s=\v_s'$ in \eqref{1.1}, $u_s$ satisfies the first order problem
\begin{equation}
\label{ii.5}
\left\{
\begin{array}{ll}
u_s'(x)+2sm'(x)u_s(x)=(c(x)-\lambda(s))\v_s(x),& x \in [-1,1],\\
u_s(y_0)=0,&
\end{array}
\right.
\end{equation}
whose solution can be easily determined by doing the change of variable
\begin{equation}
\label{ii.6}
  \v_s'(x)=u_s(x)= e^{-2sm(x)}v_s(x).
\end{equation}
Indeed, substituting in \eqref{ii.5} shows that
$$
  -2sm'(x) e^{-2sm(x)}v_s(x)+e^{-2sm(x)}v_s'(x)+2sm'(x)e^{-2sm(x)}v_s(x)=
  (c(x)-\lambda(s))\v_s(x).
$$
Thus, simplifying and rearranging terms, we find that
$$
  v_s'(x)= (c(x)-\lambda(s))\v_s(x)e^{2sm(x)} \quad \hbox{for all}\;\; x\in [-1,1].
$$
Therefore, since $v_s(y_0)=0$, it becomes apparent that
$$
   v_s(x)=\int_{y_0}^x (c(t)-\lambda(s))\v_s(t)e^{2sm(t)}\,dt.
$$
The identity \eqref{ii.3} follows readily by substituting this identity in \eqref{ii.6}.
\end{proof}

Now, we will introduce some basic concepts and notational conventions that are going to be used
through the rest of this paper. A function $h\in \mc{C}(J)$ defined on some subinterval $J$ of $[-1,1]$,
is said to satisfy $h>0$, or $h<0$, in $J$, if there exists a negligible subset $Z\subset J$, with Lebesgue measure zero, such that either $h(x)>0$ for all $x\in J\setminus Z$, or $h(x)<0$ for all $x\in J\setminus Z$, respectively. Naturally, for any other $g\in \mc{C}(J)$, it is said that $h>g$ (resp. $h<g$) if
$h-g>0$ (resp. $h-g<0$).

\begin{definition}
\label{de2.1}
For any given  $c\in \mc{C}[-1,1]$, $\l\in\R$ and $x_1\in (-1,1)$ such that $c(x_1)=\l$,
it is said that $\l$ is \emph{transversal value} of $c(x)$ at $x_1$ if $c(x)-\l$ changes sign at $x_1$, in the
sense that there exists $\eta>0$ such that either
\begin{itemize}
\item $c-\l>0$ in $[x_1-\eta,x_1)$ and $c-\l<0$ in $(x_1,x_1+\eta]$, or
\item $c-\l<0$ in $[x_1-\eta,x_1)$ and $c-\l>0$ in $(x_1,x_1+\eta]$.
\end{itemize}
\end{definition}

 Subsequently, we will assume that $c\in \mc{C}[-1,1]$ satisfies the following transversality condition
\begin{enumerate}
\item[(TC)] For every $\l\in [c_L,c_M]$, the set of $x_1\in (-1,1)$ such that
$c(x_1)=\l$ for which $\l$ is a transversal value of $c(x)$ at $x_1$ is finite.
\end{enumerate}
The next result provides us with some important properties of the eigenfunction
$\v_s$ in a certain neighborhood of any critical point, $y_0$,  according to the nodal structure
of $c(x)-\l(s)$. It is a pivotal result to ascertain the precise shape of $\v_s(x)$ in terms of
the nodal properties of  $c(x)-\l(s)$ in $[-1,1]$.

\begin{lemma}
\label{le2.3}
Suppose that $c(x)$ satisfies {\rm (TC)} and let $y_0\in [-1,1]$ be such that $\v_s'(y_0)=0$. Then, the following properties are satisfied:
\begin{enumerate}
\item[(a)] If $y_0<1$ and there exists $x_+ \in (y_0,1)$ such that $c<\l(s)$ in $(y_0,x_+)$
and $\l(s)$ is transversal to $c(x)$ at $x_+$, then
$$
  \v_s'(x) <0 \quad \hbox{for all} \;\; x\in (y_0,x_+].
$$
\item[(b)] If $-1<y_0$ and there exists $x_- \in (-1,y_0)$ such that $c<\l(s)$ in $(x_-,y_0)$
and $\l(s)$ is transversal to $c(x)$ at $x_-$, then
$$
  \v_s'(x) > 0 \quad \hbox{for all} \;\; x\in [x_-,y_0).
$$
\item[(c)]  If $y_0<1$ and there exists $x_+ \in (y_0,1)$ such that $c>\l(s)$ in $(y_0,x_+)$
and $\l(s)$ is transversal to $c(x)$ at $x_+$, then
$$
  \v_s'(x) > 0 \quad \hbox{for all} \;\; x\in (y_0,x_+].
$$
\item[(d)] If $-1<y_0$ and there exists $x_- \in (-1,y_0)$ such that $c>\l(s)$ in $(x_-,y_0)$
and $\l(s)$ is transversal to $c(x)$ at $x_-$, then
$$
  \v_s'(x) < 0 \quad \hbox{for all} \;\; x\in [x_-,y_0).
$$
\end{enumerate}
\end{lemma}

\begin{proof}
By Lemma \ref{le2.2}, we already know that
\begin{equation}
\v_s'(x)=\int_{y_0}^x(c(t)-\l(s))\v_s(t)e^{2s(m(t)-m(x))}\,dt
\end{equation}
for all $x\in [-1,1]$. Thus, Part (a) follows from the fact that $x>y_0$ if $x\in (y_0,x_+)$ and
that $c<\l(s)$ in $(y_0,x_+)$. Similarly, Part (c) follows from the hypothesis that $c>\l(s)$
in $(y_0,x_+)$.
\par
Lastly,  note that in the setting of Parts (b) and (d), since $x<y_0$,
\begin{equation}
\v_s'(x)=-\int_{x}^{y_0}(c(t)-\l(s))\v_s(t)e^{2s(m(t)-m(x))}\,dt.
\end{equation}
Thus, $\v_s'(x)<0$ if $c>\l(s)$ in $(x_-,y_0)$, while $\v_s'(x)>0$ if $c<\l(s)$ in $(x_-,y_0)$. This ends the proof.
\end{proof}

Based on Lemma \ref{le2.3}, the next result holds.

\begin{theorem}
\label{th2.1}
Suppose that $c$ satisfies the transversality condition {\rm (TC)}. Then, the set of values of $x_t \in (-1,1)$ for which $\l(s)$ is a transversal value of $c(x)$ at $x_t$ is finite, say
$$
  \mathcal{T}_s=\{x_{t,j}:\;1\leq j \leq q_s\}
$$
for some integer $q_s\geq 1$. Moreover,
\begin{enumerate}
\item[(a)]  $\v_s'(x)<0$ for all $x\in (-1,1)$ if $q_s=1$ and $c<\l(s)$ in $(-1,x_{t,1})$.
\item[(b)] $\v_s'(x)>0$ for all $x\in (-1,1)$  if $q_s=1$ and $c>\l(s)$ in $(-1,x_{t,1})$.
\item[(c)] Assume that $q_s\geq 2$. Then:
\begin{enumerate}
\item[(i)] $\v_s'(x)<0$ (resp. $\v_s'(x)>0$) for all $x\in (-1,x_{t,1}]$ if $c<\l(s)$ (resp. $c>\l(s)$) in $(-1,x_{t,1})$.
\item[(ii)] $\v_s'(x)<0$ (resp. $\v_s'(x)>0$) for all $x\in [x_{t,q_s},1)$ if $c>\l(s)$ (resp. $c<\l(s)$) in $(x_{t,q_s},1)$.
\item[(iii)] For any given $j\in\{1,...,q_s-1\}$,  should there exists some point $y_{0,j}\in (x_{t,j},x_{t,j+1})$ such that $\v_s'(y_{0,j})=0$, then $y_{0,j}$ is unique, and
\begin{equation}
\label{ii.7}
\left\{\begin{array}{l} \v_s'(x)>0 \;\;\hbox{for all}\;\; x\in[x_{t,j},y_{0,j}),\\
\v_s'(x)<0 \;\;\hbox{for all}\;\; x\in(y_{0,j},x_{t,j+1}],\end{array}\right.
\end{equation}
if $c<\l(s)$ in $(x_{t,j},x_{t,j+1})$, whereas
\begin{equation}
\label{ii.8}
\left\{\begin{array}{l}
\v_s'(x)<0 \;\;\hbox{for all}\; \; x\in[x_{t,j},y_{0,j}),\\
\v_s'(x)>0 \;\;\hbox{for all}\; \; x\in(y_{0,j},x_{t,j+1}],\end{array}\right.
\end{equation}
if $c>\l(s)$ in $(x_{t,j},x_{t,j+1})$.  Therefore, $y_{0,j}$ is the unique critical point of the eigenfunction $\v_s$ in $[x_{t,j},x_{t,j+1}]$. Furthermore, $y_{0,j}$ is a global maximum if $c<\l(s)$ in $(x_{t,j},x_{t,j+1})$, and a global minimum if $c>\l(s)$ in $(x_{t,j},x_{t,j+1})$.
\end{enumerate}
\end{enumerate}
\end{theorem}

\begin{proof}
The first assertion of the statement is a direct consequence of the transversality condition (TC). Suppose
$q_s=1$ and $c<\l(s)$ in $(-1,x_{t,1})$. Then, $c>\l(s)$ in $(x_{t,1},1)$. Thus, by applying Lemma \ref{le2.3}(a) with $y_0=-1$ and $x_+=x_{t,1}$, it is apparent that $\v_s'(x)<0$ for all $x\in (-1,x_{t,1}]$. Similarly, by Lemma \ref{le2.3}(d) with $x_-=x_{t,1}$ and $y_0=1$,
it becomes apparent that $\v'_s(x)<0$ for all $x\in [x_{t,1},1)$, because $c>\l(s)$ there in. This concludes the proof of Part (a).
\par
Now, suppose that $q_s=1$ and $c>\l(s)$ in $(-1,x_{t,1})$. Then, $c<\l(s)$ in $(x_{t,1},1)$. Thus, by applying Lemma \ref{le2.3}(c) with $y_0=-1$ and $x_+=x_{t,1}$, it is apparent that $\v_s'(x)>0$ for all $x\in (-1,x_{t,1}]$. Similarly, by Lemma \ref{le2.3}(b) with $x_-=x_{t,1}$ and $y_0=1$,
it becomes apparent that $\v'_s(x)>0$ for all $x\in [x_{t,1},1)$, because $c<\l(s)$ there in, which ends  the proof of Part (b).
\par
For the rest of this proof, we assume that $q_s\geq 2$. Suppose that
$c<\l(s)$ in the interval $(-1,x_{t,1})$. Then, by Lemma \ref{le2.3}(a) with $y_0=-1$ and $x_+=x_{t,1}$,
we find that $\v_s'(x) <0$ for all $x\in (-1,x_{t,1}]$. Similarly, when $c>\l(s)$  in $(-1,x_{t,1})$,
the fact that $\v_s'(x) >0$ for all $x\in (-1,x_{t,1}]$ is a direct consequence of Lemma \ref{le2.3}(c).
This ends the proof of Part (c)(i). Part (c)(ii) follows readily by applying  Lemma \ref{le2.3}(b) and (d)  with $y_0=1$ and $x_-=x_{t,q_s}$.
\par
Lastly, suppose that, for some $j\in \{1,...,q_s-1\}$, there is some $y_{0,j}\in (x_{t,j},x_{t,j+1})$
such that $\v'_s(y_{0,j})=0$, and that $c<\l(s)$ in $(x_{t,j},x_{t,j+1})$. Then, owing to \eqref{ii.3},
it is apparent that \eqref{ii.7} holds. Similarly, when $c>\l(s)$ in $(x_{t,j},x_{t,j+1})$, \eqref{ii.8} holds. Since \eqref{ii.7} and \eqref{ii.8} entail the uniqueness of the critical point of $\v_s$ in $(x_{t,j},x_{t,j+1})$, the proof is complete.
\end{proof}

Naturally, when $c(x)$ is monotone, then $q_s=1$ and hence, by Theorem \ref{th2.1}(a), $\v_s'(x)<0$
for all $x\in (-1,1)$ if $c$ is non-decreasing in $[-1,1]$, whereas $\v_s'(x)>0$
for all $x\in (-1,1)$ if $c$ is non-increasing in $[-1,1]$.
\par
According to Part (c)(iii), $\v_s$ has, at most, a unique critical point in
the interval $[x_{t,j},x_{t,j+1}]$. Whether, or not, $\v_s$ admits some critical point in the interval
$(x_{t,j},x_{t,j+1})$ depends on the particular choices of $c(x)$ and $m(x)$. The value of $\l(s)$ might severely affect the number of admissible critical points of the normalized eigenfunction $\v_s$, of course.  Anyway, the oscillatory behavior of $\v_s$ can be ascertained with the following argument.
\par
According to Theorem \ref{th2.1}(a,b), the eigenfunction $\v_s$ is strictly monotone if $q_s=1$.
Now, suppose that $q_s=2$ and that, e.g., $c>\l(s)$ in $(-1,x_{t,1})$. Then, $c<\l(s)$ in
$(x_{t,1},x_{t,2})$, and $c>\l(s)$ in $(x_{t,2},1)$. Thus, by Theorem \ref{th2.1}(c)(i), we have that
$\v_s'(x)>0$ for all $x\in (-1,x_{t,1}]$. Similarly, owing to Theorem \ref{th2.1}(c)(ii),
$\v_s'(x)<0$ for all $x\in [x_{t,2},1)$. In particular,
$$
   \v_s'(x_{t,1})>0 \;\; \hbox{and}\;\;  \v_s'(x_{t,2})<0.
$$
Thus, there exists $y_{0,1}\in (x_{t,1},x_{t,2})$ such that $\v_s'(y_{0,1})=0$. Therefore,
by Theorem \ref{th2.1}(c), $y_{0,1}$ is the unique critical point
of $\v_s$ in $[x_{t,1},x_{t,2}]$ and
$$
  \v_s'(x) \left\{ \begin{array}{l}  > 0 \;\;\hbox{for all}\;\; x\in [x_{t,1},y_{0,1}),\\
  < 0 \;\;\hbox{for all}\;\; x\in (y_{0,1},x_{t,2}).\end{array}\right.
$$
In particular,
$$
  \v_s(y_{0,1})=\|\v_s\|_\infty =1.
$$
This argument can be easily adapted to show that, as soon as $q_s=2$ and $c<\l(s)$ in $(-1,x_{t,1})$,
there exists a unique $y_{0,1}\in (x_{t,1},x_{t,2})$ such that
$$
  \v_s'(x) \left\{ \begin{array}{l} < 0 \;\;\hbox{for all}\;\; x\in (-1,y_{0,1}),\\
  > 0 \;\;\hbox{for all}\;\; x\in (y_{0,1},1).\end{array}\right.
$$
In particular,
$$
   \v_s(y_{0,1})=(\v_s)_L=\min_{[-1,1]}\v_s.
$$
Now, suppose that $q_s=3$ and that $c>\l(s)$ in $(-1,x_{t,1})$. Then, $c<\l(s)$ in
$(x_{t,1},x_{t,2})$, $c>\l(s)$ in $(x_{t,2},x_{t,3})$, and $c<\l(s)$ in $(x_{t,3},1)$. Thus, due to
Theorem \ref{th2.1},
\begin{equation*}
  \v_s'(x) \left\{ \begin{array}{l} >0 \;\;\hbox{for all}\;\; x\in (-1,x_{t,1}], \\
  > 0 \;\;\hbox{for all}\;\; x\in [x_{t,3},1).\end{array}\right.
\end{equation*}
Suppose, in addition, that there exists $y_{0,1}\in (x_{t,1},x_{t,2})$ such that $\v_s'(y_{0,1})=0$. Then,
thanks to Theorem \ref{th2.1}, $y_{0,1}$ is unique and, since $c<\l(s)$ in $(x_{t,1},x_{t,2})$, we find that
\begin{equation*}
\v_s'(x) \left\{\begin{array}{l} >0 \;\;\hbox{for all}\;\; x\in[x_{t,1},y_{0,1}),\\
<0 \;\;\hbox{for all}\;\; x\in(y_{0,1},x_{t,2}].\end{array}\right.
\end{equation*}
Furthermore, since $\v_s'(x_{t,2})<0$ and $\v_s'(x_{t,3})>0$, there exists $y_{0,2}\in (x_{t,2},x_{t,3})$
such that $\v_s'(y_{0,2})=0$, and, due to Theorem \ref{th2.1}, $y_{0,2}$ is unique and
\begin{equation*}
\v_s'(x) \left\{\begin{array}{l} <0 \;\;\hbox{for all}\;\; x\in[x_{t,2},y_{0,2}),\\
>0 \;\;\hbox{for all}\;\; x\in(y_{0,2},x_{t,3}].\end{array}\right.
\end{equation*}
Therefore, $y_{0,2}$ it is a local minimum in $[x_{t,2},x_{t,3}]$.
\par
Another admissible situation in case $q_s=3$ is that $\v_s'(x)>0$ for all $x\in(-1,x_{t,2}]$. In such case,
since $\v_s'(x_{t,3})>0$ and $\v_s$ admits, at most, a unique critical point in $(x_{t,2},x_{t,3})$, it is
apparent that $\v_s'(x)>0$ for all $x\in (-1,1)$ and hence, $\v_s$ is strictly increasing, much like occurs
in the case $q_s=1$.
\par
At this stage of the discussion the reader should be able to infer very easily what are the admissible profiles of $\v_s$ when $q_s=3$ and $c<\l(s)$ in $(-1,x_{t,1})$, and, more generally, when
$q_s\geq 4$. Essentially, when $q_s$ is odd, in the end intervals, $(-1,x_{t,1})$ and $(x_{t,q_s},1)$, $\v_s$
has identical character, while in the even interior intervals $\v_s$ admits up to $\frac{q_s-1}{2}$ maxima
separated by identical number of minima. Similarly, when $q_s$ is even, in the end intervals, $(-1,x_{t,1})$ and $(x_{t,q_s},1)$, $\v_s$ has opposite character, while in the $q_s-1$ central intervals $\v_s$ exhibits
up to $q_s-1$ critical points, alternating the local maxima with the local minima, which are unique
on each of the nodal intervals of $c-\l(s)$ where they exist.

\section{Stabilization of $\v_s$ as $s \uparrow \infty$ }

\noindent The main result of this section can be stated as follows.

\begin{theorem}
\label{th3.1}
Suppose $m$ satisfies {\rm (Hm)}. Then, for every $x\in [-1,1]$,
\begin{equation}
\label{iii.1}
\lim_{s \rightarrow \infty}\v_s'(x) =0\quad \hbox{and}\quad \lim_{s \rightarrow \infty}
\v_s(x)=1,
\end{equation}
uniformly in $[-1,1]$.
\end{theorem}
\begin{proof}
By the assumptions on $m$, there exists $x_0\in (-1,1)$ such that $m(x_0)=m_M$,
$m'(x) \geq 0$ for all $x \in [-1,x_0)$, $m'(x) \leq 0$ for all $x \in (x_0,1]$ and $m$ possesses finitely many critical points in $[-1,1]$. We first prove that
\begin{equation}
\label{iii.17}
\lim_{s\to \infty}\|\v_s'\|_\infty =0.
\end{equation}
\noindent
We will divide the proof of \eqref{iii.17} in several cases, according to the number of critical points of $m(x)$.
\vspace{0.2cm}
\par
\par
\noindent \emph{Case 1.} In this case, we assume that $m'(x)>0$ for all $x \in [-1,x_0)$ and $m'(x)<0$ for all $x \in (x_0,1]$. Thus, $x_0$ is the unique critical point of $m(x)$ in $[-1,1]$. Next, we pick $\e>0$, fix $\d=\d(\e)$ such that
\begin{equation}
\label{iii.2}
0<\d <\min \left\{ \frac{\varepsilon}{4 \Vert c \Vert_{\infty}} ,1+x_0,1-x_0\right\},
\end{equation}
and divide the proof of \eqref{iii.17} into four steps.
\par
\vspace{0.2cm}

\noindent \emph{Step 1.} There exists $s_1=s_1(\varepsilon)>0$ such that
\begin{equation}
\label{iii.3}
\vert \v'_s(x) \vert< \e \;\; \hbox{for all}\;\; x \in [-1,x_0-\d] \;\; \hbox{and} \;\;  s>s_1(\e).
\end{equation}
Indeed, let $x \in (-1,x_0-\d]$ be and set
$$
m'_{L,\d}:=\min_{[-1,x_0-\d]} m'>0.
$$
Then, since $\|\v_s\|_\infty=1$, it follows from  \eqref{ii.3} with $y_0=-1$ and Lemma \ref{le2.1} that
\begin{equation}
\label{iii.4}
\vert \v_s'(x) \vert \leq \int_{-1}^x \vert c(t)-\l(s) \vert \v_s(t)e^{2s(m(t)-m(x))}dt \leq 2 \Vert c \Vert_{\infty}\int_{-1}^{x}e^{-2s(m(x)-m(t))}\,dt.
\end{equation}
On the other hand, thanks to the mean value theorem, for every $t\in (-1,x)$, there exists
$$
   \xi \in (t,x) \subset [-1,x_0-\d]
$$
such that
$$
   m(x)-m(t)=m'(\xi)(x-t)\geq m'_{L,\d}(x-t).
$$
Thus, for every $x \in (-1,x_0-\d]$,
\begin{align*}
\int_{-1}^{x}e^{-2s(m(x)-m(t))}\,dt & \leq \int_{-1}^x e^{-2sm'_{L,\d}(x-t)}\,dt \\ & = \frac{1-e^{-2s m'_{L,\d}(x+1)}}{2s m'_{L,\d}}
\leq \frac{1-e^{-2s m'_{L,\d}(x_0-\d+1)}}{2sm'_{L,\d}}.
\end{align*}
Hence, substituting in \eqref{iii.4}, we find that, for every $x \in (-1,x_0-\d]$,
\begin{equation}
\label{iii.5}
\vert \v'_s(x) \vert \leq \frac{\Vert c \Vert_{\infty}}{m'_{L,\d}}
\frac{1-e^{-2s m'_{L,\d}(x_0-\d+1)}}{s}.
\end{equation}
Therefore, since $m'_{L,\d}>0$ and, due to \eqref{iii.2},  $\d<1+x_0$, we find from \eqref{iii.5}
that there exists $s_1(\e)>0$ such that $\vert \v_s'(x) \vert < \varepsilon$ for all $s>s_1(\varepsilon)$ and $x \in (-1,x_0-\delta]$. So, since $\v'_s(-1)=0$, \eqref{iii.3} holds.
\par
\vspace{0.2cm}

\noindent \emph{Step 2.} There exists $s_2=s_2(\varepsilon)>s_1(\varepsilon)$ such that
\begin{equation}
\label{iii.6}
\vert \v_s'(x) \vert< \e \;\; \hbox{for all} \;\; x \in [x_0-\d,x_0] \;\; \hbox{and} \;\;   s>s_2(\e).
\end{equation}
Indeed, by applying \eqref{ii.3} with $y_0=-1$ and, since $x \in [x_0-\d,x_0]$, we find that
\begin{equation}
\label{iii.7}
\begin{split}
\v_s'(x) & = \int_{-1}^{x}(c(t)-\l(s)) \v_s(t) e^{2s(m(t)-m(x))}\,dt\\ & =
\v_s'(x_0-\d) +\int_{x_0-\d}^{x}(c(t)-\l(s)) \v_s(t) e^{2s(m(t)-m(x))}\,dt.
\end{split}
\end{equation}
By Step 1, we already know that there exists $s_2(\varepsilon)>s_1(\varepsilon)>0$ such that
\begin{equation}
\label{iii.8}
|\v_s'(x_0-\d)|<\frac{\varepsilon}{2}\;\; \hbox{for all}\;\; s>s_2(\e).
\end{equation}
On the other hand, since  $m'(x)>0$ for all $x \in (-1,x_0)$,  we have that
$$
   m(t) \leq m(x)\;\; \hbox{for all}\;\; t \in [x_0-\d, x] \subset (-1,x_0).
$$
Thus, $e^{2s(m(t)-m(x))}\leq 1$ for all $s>0$ and, hence, \eqref{iii.7} and \eqref{iii.8} guarantee that, for every  $x \in [x_0-\d,x_0]$,
\begin{equation*}
  |\v_s'(x)|\leq \frac{\e}{2}+ 2 \Vert c \Vert_{\infty} \int_{x_0-\d}^x e^{2s(m(t)-m(x))}\,dt\leq
  \frac{\e}{2}+ 2 \Vert c \Vert_{\infty} \d< \e,
\end{equation*}
by the choice of $\d$ in \eqref{iii.2}. Therefore, \eqref{iii.6} holds.
\par
\vspace{0.2cm}

\noindent\emph{Step 3.} There exists $s_3=s_3(\e)>s_2(\e)$ such that
\begin{equation}
\label{iii.9}
\vert \v_s'(x) \vert< \e \;\; \hbox{for all}\;\; x \in [x_0+\d,1] \;\; \hbox{and} \;\; s>s_3(\e).
\end{equation}
Indeed, let $x \in [x_0+\d,1)$ be and set
$$
m'_{M,\d}:=\max_{[x_0+\d,1]}m' <0.
$$
Adapting the argument of Step 1, it follows from \eqref{ii.3} with $y_0=1$ that
\begin{equation}
\label{iii.10}
\vert \v_s'(x) \vert \leq \int_{x}^1 \vert \l(s)-c(t) \vert \v_s(t)e^{2s(m(t)-m(x))}dt \leq 2 \Vert c \Vert_{\infty}\int_{x}^{1}e^{2s(m(t)-m(x))}\,dt.
\end{equation}
Moreover, by the mean value theorem,
\begin{equation}
\label{iii.11}
\int_{x}^{1}e^{2s(m(t)-m(x))}\,dt \leq \int_{x}^1 e^{2sm_{M,\d}^{'}(t-x)}\,dt=
\frac{e^{2sm_{M,\d}^{'}(1-x)}-1}{2sm_{M,\d}^{'}}.
\end{equation}
Therefore, since $m_{M,\d}^{'}<0$, it follows from \eqref{iii.10} and \eqref{iii.11} that
\begin{equation}
\label{iii.12}
\vert \v_s'(x) \vert \leq \frac{\Vert c \Vert_{\infty}}{\vert m'_{M,\d}\vert} \frac{1-e^{2s m'_{M,\d}(1-x)}}{s}\leq \frac{\Vert c \Vert_{\infty}}{\vert m_{M,\d}^{'}\vert} \frac{1-e^{2s m_{M,\d}^{'}(1-x_0-\d)}}{s}.
\end{equation}
Consequently, since $\v_s'(1)=0$ and $\d<1-x_0$, because of the choice of $\d$ in \eqref{iii.2}, it follows from \eqref{iii.12}  that there exists $s_3(\e)>s_2(\e)$ such that \eqref{iii.9} holds.
\par
\vspace{0.2cm}

\noindent \emph{Step 4.} There exists $s_4=s_4(\e)>s_3(\e)>0$ such that
\begin{equation}
\label{iii.13}
\vert \v_s'(x) \vert< \e \;\; \hbox{for all} \;\; x \in [x_0,x_0+\d] \;\;\hbox{and} \;\; s>s_4(\e).
\end{equation}
Indeed, owing to \eqref{ii.4} and since $x \in [x_0,x_0+\d]$, we have that
\begin{equation}
\label{iii.14}
\v_s'(x)= \int_{x}^{x_0+\d}(\l(s)-c(t)) \v_s(t) e^{2s(m(t)-m(x))}\,dt+\int_{x_0+\d}^{1}(\l(s)-c(t)) \v_s(t) e^{2s(m(t)-m(x))}\,dt.
\end{equation}
According to Step 3, there exists $s_4=s_4(\e)>s_3(\e)>0$ such that
\begin{equation}
\label{iii.15}
\left \vert \int_{x_0+\d}^{1}(\l(s)-c(t)) \v_s(t) e^{2s(m(t)-m(x))} \,dt \right \vert<\frac{\varepsilon}{2} \;\; \hbox{for all} \;\; s>s_4.
\end{equation}
On the other hand, since $m'(x)<0$ for all $x \in (x_0,1)$,
$$
   m(x)>m(t)\;\; \hbox{for all}\;\; t \in [x,x_0+\d] \subset (x_0,1).
$$
Thus, since $\|\v_s\|_\infty =1$, adapting the argument of Step 2,  it follows from Lemma \ref{le2.1} and \eqref{iii.2}, that
\begin{equation}
\label{iii.16}
\left \vert \int_x^{x_0+\d}(\l(s)-c(t))\v_s(t) e^{2s(m(t)-m(x))} \,dt\right \vert
\leq 2 \Vert c \Vert_{\infty} \int_x^{x_0+\d}e^{2s(m(t)-m(x))}\,dt\leq 2 \Vert c \Vert_{\infty}\d< \frac{\varepsilon}{2}.
\end{equation}
Lastly, \eqref{iii.13} follows from \eqref{iii.14}, \eqref{iii.15} and \eqref{iii.16}.
This ends the proof in \emph{Case 1}.
\vspace{0.2cm}
\par
\noindent \emph{Case 2.}  In this case, we assume that $m$ has two critical points, $x_0$ and $x_c$,  in $[-1,1]$, with $x_c \neq x_0$, and $x_c\in (-1,1)$. Then, either
$x_c < x_0$, or $x_c > x_0$. Pick $\e>0$ and fix
$\d=\d(\e)$ such that
\begin{equation}
\label{iii.20bis}
0<\d <\min \left\{ \frac{\varepsilon}{8 \Vert c \Vert_{\infty}} ,1+x_c,1-x_c\right\}.
\end{equation}
Suppose that $x_c<x_0$. Then, taking into account that
$$
\min_{[-1,x_c-\d]}m'>0,
$$
and adapting the proof of \emph{Step 1} in \emph{Case 1}, it readily follows the existence of $s_1=s_1(\e)>0$ such that
\begin{equation}
\label{iii.21}
\vert \v'_s(x) \vert < \e \quad \mbox{for all} \quad x\in  [-1,x_c-\d] \quad \mbox{and} \quad s>s_1(\e).
\end{equation}
Now, we will show that there exists $s_2=s_2(\e)>s_1(\e)>0$ such that
\begin{equation}
\label{iii.22}
\vert \v'_s(x) \vert < \e \quad \mbox{for all} \quad x\in  [x_c-\d,x_c+\d] \quad \mbox{and} \quad s>s_2(\e).
\end{equation}
Indeed, by applying \eqref{ii.3} with $y_0=-1$, we find that, for every  $x \in [x_c-\d,x_c+\d]$, \begin{equation}
\label{iii.23}
\begin{split}
\v_s'(x) & = \int_{-1}^{x}(c(t)-\l(s)) \v_s(t) e^{2s(m(t)-m(x))}\,dt\\ & =
\v_s'(x_c-\d) +\int_{x_c-\d}^{x}(c(t)-\l(s)) \v_s(t) e^{2s(m(t)-m(x))}\,dt.
\end{split}
\end{equation}
By \eqref{iii.21}, there exists $s_2=s_2(\e)>s_1(\e)>0$ such that
\begin{equation}
\label{iii.24}
\vert \v_s'(x_c-\d) \vert < \frac{\e}{2} \quad \mbox{for all} \quad s>s_2(\e).
\end{equation}
On the other hand, since $m'(x) \geq 0$ for all $x \in [x_c-\d,x_c+\d]$, we have that
$$
m(t) \leq m(x) \quad \mbox{for all} \quad t \in [x_c-\d,x] \subset [x_c-\d,x_c+\d].
$$
Thus, for every $s>0$, we find from \eqref{iii.20bis} that
\begin{equation}
\label{jlg}
  \left| \int_{x_c-\d}^{x}(c(t)-\l(s)) \v_s(t) e^{2s(m(t)-m(x))}\,dt\right|\leq
  4\|c\|_\infty\d  <\frac{\e}{2}.
\end{equation}
Consequently, \eqref{iii.22} holds from \eqref{iii.23}, \eqref{iii.24} and \eqref{jlg}.
\par
Finally, since $m'(x_c+\d)>0$ and $m'(1)<0$, the existence of $s_3=s_3(\e)>s_2(\e)>0$ such that $\vert \v_s'(x) \vert < \e$ for all $x \in [x_c+\d,1]$ can be obtained by adapting to the interval $[x_c+\d,1]$ the argument of the proof of  \emph{Case 1} delivered in the interval $[-1,1]$. Note that one should divide the interval $[x_c+\d,1]$ into  the subintervals $[x_c+\d,x_0-\d^*]$, $[x_0-\d^*,x_0]$, $[x_0,x_0+\d^*]$ and $[x_0+\d^*,1]$ for  sufficiently small $\d^*>0$.
This ends the proof of \emph{Case 2} when $x_c<x_0$. By making some (obvious) changes, this proof can be easily adapted to cover the case when $x_c>x_0$. The proof of \emph{Case 2} is complete.
\par
\vspace{0.2cm}

\noindent \emph{Case 3}. In this case, we assume that $m'(\pm 1)\neq 0$ and that the set $\mathscr{C}$ of critical points of $m$ in $(-1,1)$ is
\begin{equation}
\label{iii.220}
\mathscr{C}=C_\ell \cup \left\{ x_0 \right\} \cup C_r,
\end{equation}
with either $C_\ell=\emptyset$, or $C_\ell=\left\{ x_{\ell_j,c}:\; 1\leq j\leq p\right\}$
for some $p\in \mathbb{N}$, and $C_r=\emptyset$, or $C_r=\left\{ x_{r_k,c}:\; 1\leq k\leq q\right\}$ for some $q\in \mathbb{N}$, and  $\mathscr{C}\setminus\{x_0\} = C_\ell \cup C_r \neq \emptyset$, with the convention that $x_{\ell_j,c}<x_{\ell_{j+1},c}$ and
$x_{r_{k+1},c}<x_{r_k,c}$. So, in the most general case when $C_\ell$ and $C_r$ are nonempty, we have that
\begin{equation}
\label{iii.223}
-1<\cdots<x_{\ell_j,c}<x_{\ell_{j+1},c}<\cdots<x_0<\cdots<x_{r_{k+1},c}<x_{r_k,c}<\cdots<1.
\end{equation}
Pick $\e>0$ and take $\d \in (0,\frac{\e}{8 \Vert c\Vert_{\infty}})$ sufficiently small so that
\begin{equation}
\label{iii.224}
(x_{p_1}-\d,x_{p_1}+\d) \cup (x_{p_2}-\d,x_{p_2}+\d) \subset (-1,1), \qquad (x_{p_1}-\d,x_{p_1}+\d) \cap (x_{p_2}-\d,x_{p_2}+\d) =\emptyset,
\end{equation}
for any pair of critical points $x_{p_1}, x_{p_2}\in \mathscr{C}$. Then, the proof of the existence of $s=s(\e)>0$ such that $\vert \v_s'(x) \vert <\e$ for all $x \in [-1,1]$ and $s>s(\e)>0$ can be accomplished as follows:
\begin{enumerate}
\item[(i)] arguing in each of the intervals $[x_{\ell_{j},c}-\d,x_{\ell_{j},c}+\d]$, $1\leq j\leq p$, as
in the proof of Case 2 in the interval $[x_c-\d,x_c+\d]$ when $x_c<x_0$, if $C_\ell$ is nonempty;

\item[(ii)] arguing in each of the intervals $[x_{r_{k},c}-\d,x_{r_{k},c}+\d]$, $1\leq k\leq q$, as
in the proof of Case 2 in the interval $[x_c-\d,x_c+\d]$ when $x_c>x_0$, if $C_r$ is nonempty;

\item[(iii)] arguing in the interval $[x_0-\d,x_0+\d]$ as in the proofs of Steps 2 and 4 of \emph{Case 1};

\item[(iv)] arguing in the complement of the set
$$
   \cup_{j=1}^p (x_{\ell_{j},c}-\d,x_{\ell_{j},c}+\d)\cup (x_0-\d,x_0+\d)\cup
     \cup_{k=1}^q (x_{r_{k},c}-\d,x_{r_{k},c}+\d)
$$
in $[-1,1]$ as, e.g.,  in the proof of \emph{Case 1} in $[-1,x_0-\d]\cup[x_0+\d,1]$.

\end{enumerate}
\par

\vspace{0.2cm}

\noindent \emph{Case 4}. In this case, besides the conditions of \emph{Case 3}, we assume, in addition, that $m'(-1)=0$, and/or $m'(1)=0$. Suppose $m'(-1)=0$. We claim that, for every $\e>0$, there exists $\d=\d(\e)>0$ such that, for sufficiently large $s>0$,
\begin{equation}
\label{iii.26}
\vert \v_s'(x) \vert <\e \qquad \mbox{for all} \quad x \in [-1,-1+\d].
\end{equation}
Indeed, choose a sufficiently small $\d>0$  so that $m$ cannot admit a critical point in $(-1,-1+\d]$ and \eqref{iii.224} holds. Then, by applying \eqref{ii.3} with $y_0=-1$ we find that, for every $x  \in [-1,-1+\d]$,
\begin{equation}
\label{ii.270}
\v_s'(x)=\int_{-1}^x(c(t)-\l(s))\v_s(t)e^{2s(m(t)-m(x))}\,dt.
\end{equation}
Since $m'(x) \geq 0$ for all  $x \in [-1,-1+\d]$, it is apparent that
\begin{equation}
\label{iii.28}
m(t)-m(x)\leq 0 \quad \mbox{for all} \quad t \in [-1,x]\subset [-1,-1+\d].
\end{equation}
Thus, by Lemma \ref{le2.1}, it follows from  \eqref{ii.270} and \eqref{iii.28} that
$$
\vert \v_s'(x) \vert \leq 2 \Vert c \Vert_{\infty}\d.
$$
Shortening $\d$, if necessary, yields \eqref{iii.26}. A similar estimate holds in the interval $[1-\d,1]$
if $m'(1)=0$. The rest of the proof can be completed arguing as in the proof of \emph{Case 3}. This ends the proof of \eqref{iii.17}.
\par
\vspace{0.2cm}
To conclude the proof of the theorem, it remains to show that
\begin{equation}
\label{iii.29}
\lim_{s \rightarrow \infty}\v(s)=1 \qquad \mbox{uniformly in  } [-1,1].
\end{equation}
Indeed, since $\Vert \v_s \Vert_{\infty}=1$, there exists
$x_s \in [-1,1]$ such that $\v_s(x_s)=1$. Thus,
$$
\v_s(x)=\v_s(x_s)+\int_{x_s}^x \v'_s(\xi)\,d\xi=1+\int_{x_s}^x \v'_s(\xi)\,d\xi \;\;
\hbox{for all} \;\; x\in [-1,1].
$$
Hence,
\begin{equation}
\label{iii.18}
\vert \v_s(x)-1 \vert \leq \int_{\min\{x_s,x\}}^{\max\{x_s,x\}} \vert \v'_s(\xi) \vert \,d\xi \leq   \int_{-1}^{1} \vert \v'_s(\xi) \vert\,d\xi.
\end{equation}
Therefore,  \eqref{iii.29} holds from \eqref{iii.17} and \eqref{iii.18}. The proof is complete.

\end{proof}

\section{Asymptotic stabilization of $\l(s)$ as $s \uparrow \infty$ }

\noindent The main result of this section can be stated as follows.

\begin{theorem}
\label{th4.1}
Suppose $m(x)$ satisfies {\rm (Hm)}, and
\begin{enumerate}
\item[{\rm (Hc)}] the function $c(x)$ satisfies some of the following conditions:
\begin{enumerate}
\item[{\rm (a)}] there exists $y_0\in (-1,1)$ such that $c$ is increasing (resp. decreasing) in $[-1,y_0)$ and decreasing (resp. increasing) in $(y_0,1]$,

\item[{\rm (b)}] $c$ is monotone (either increasing or decreasing) in $[-1,1]$.
\end{enumerate}
\end{enumerate}
Then,
\begin{equation}
\label{4.1}
    \lim_{s \uparrow \infty}\l(s)=c(x_0).
\end{equation}
\end{theorem}
\begin{proof}
We will show  that
\begin{equation}
\lim_{n \rightarrow \infty} \l(s_n)=c(x_0).
\end{equation}
for every sequence of real numbers $\{ s_n\}_{n\geq 1}$ such that
\begin{equation}
\label{4.2}
\lim_{n \rightarrow \infty }s_n=\infty.
\end{equation}
Subsequently, we fix a sequence satisfying \eqref{4.2} and set
$$
   \l_n:=\l(s_n), \qquad \v_n:=\v_{s_n},\qquad n\geq 1.
$$
By Lemma \ref{le2.1}, $\l_n \in (c_L,c_M)$ for all $n\geq 1$. Thus, there is a subsequence of $\l_n$,
relabeled by $n$, such that
\begin{equation}
\label{4.3}
\lim_{n \rightarrow \infty}\l_n=\lambda_{\infty}\in [c_L,c_M].
\end{equation}
By the continuity of $c$, there exists $x_{\infty}\in [-1,1]$ such that
\begin{equation}
\label{4.4}
\l_{\infty}=c(x_{\infty}).
\end{equation}
To complete the proof, it suffices to show that $x_\infty=x_0$. Indeed, as we can repeat the same
argument along any subsequence, this implies \eqref{4.1}.  Our argument will proceed by contradiction. So, assume that
\begin{equation}
\label{4.5}
  \l_\infty=c(x_\infty)\neq c(x_0).
\end{equation}
We will reach a contradiction by steps, according to the structure of the roots of  $c(x) = \l_{\infty}$ in $(-1,1)$.  According to  \eqref{4.3} and the profile of $c$, either $c(x) = \l_{\infty}$ possesses a unique transversal root in $(-1,1)$, or it possesses two transversal roots, or it possesses a unique root which is not transversal, or it has no roots in $(-1,1)$. Next, we will deliver the technical details of the proof when $c(x)$ satisfies (Hc)(a). The proof in the simplest case when $c(x)$ satisfies (Hc)(b) is omitted by repetitive.
\par
\vspace{0.2cm}
\noindent\emph{Case 1.} Suppose that $c(x) = \l_{\infty}$ has a unique transversal root in $(-1,1)$, $x_{\infty}$. Thanks to \eqref{4.4} and \eqref{4.5},  $x_{\infty} \neq x_0$. Thus, either
\begin{equation}
\label{4.6}
x_{\infty}<x_0,
\end{equation}
or
\begin{equation}
\label{4.7}
x_{\infty}>x_0.
\end{equation}
Moreover, either
\begin{equation}
\label{4.8}
c<\l_{\infty} \;\;  \hbox{in} \;\; [-1,x_{\infty}) \;\; \hbox{and} \;\;
c>\l_{\infty} \;\; \hbox{in} \;\; (x_{\infty},1],
\end{equation}
or
\begin{equation}
\label{4.9}
c>\l_{\infty} \;\;  \hbox{in} \;\; [-1,x_{\infty}) \;\; \hbox{and} \;\;
c<\l_{\infty} \;\; \hbox{in} \;\; (x_{\infty},1].
\end{equation}
By \eqref{4.3}, the continuity of $c(x)$ and the transversality of $x_{\infty}$,
there is a sequence, $\{x_{n}\}_{n\geq 1}$, in $(-1,1)$ such that
\begin{equation}
\label{4.10}
\lim_{n \rightarrow \infty}x_n=x_{\infty}\;\;\hbox{and}\;\; c(x_n)=\l_n\;\;\hbox{for all}\;\; n\geq 1.
\end{equation}
Moreover,  $x_n$ is the unique transversal root of $c(x)=\l_n$ for sufficiently large $n$.
\par
Assume that \eqref{4.6} and \eqref{4.8} hold. Then, by \eqref{4.6} and \eqref{4.10}, there exists
$n_0 \in \mathbb{N}$ such that
\begin{equation*}
x_n<x_0 \;\; \hbox{for all}\;\; n\geq n_0.
\end{equation*}
Moreover, thanks to \eqref{4.10} and \eqref{4.8}, we have that, for sufficiently large $n$, say
$n_1\geq n_0$,
\begin{equation}
\label{4.11}
c<\l_n \;\; \hbox{in} \;\;  [-1,x_n) \;\; \hbox{and} \;\;
c>\l_n \;\; \hbox{in} \;\; (x_n,1].
\end{equation}
Thanks to (Hm), we have that
\begin{equation}
\label{4.12}
m'\gneq 0 \;\; \hbox{in} \;\; [x_n,x_0]\;\; \hbox{for all}\;\; n\geq n_1.
\end{equation}
Moreover, by Theorem \ref{th2.1}(a), \eqref{4.11} entails
\begin{equation}
\label{4.13}
\v_n'(x)<0 \quad \hbox{for all} \;\; x \in (-1,1).
\end{equation}
On the other hand, integrating the differential equation of \eqref{1.1} in  $[x_n,x_0]$, we find that
\begin{equation}
\label{4.14}
\v_n'(x_n)-\v_n'(x_0) -2s_n\int_{x_n}^{x_0}m'(x)\v_n'(x)\,dx=\int_{x_n}^{x_0}(\l_n-c(x))\v_n(x)\,dx
\end{equation}
for all $n \geq n_1$. Thus, letting $n \rightarrow \infty$ in the right hand side of \eqref{4.14}, it follows from Theorem \ref{th3.1}, \eqref{4.3}, \eqref{4.8} and \eqref{4.10} that
\begin{equation}
\label{4.15}
\lim_{n \rightarrow \infty}\int_{x_n}^{x_0}(\l_n-c(x))\v_n(x)\,dx=\int_{x_{\infty}}^{x_0}(\l_{\infty}-c(x))\v_{\infty}(x)\,dx
=\int_{x_{\infty}}^{x_0}(\l_{\infty}-c(x))\,dx<0.
\end{equation}
Moreover,  also by Theorem \ref{th3.1},
\begin{equation}
\label{4.16}
\lim_{n \rightarrow \infty}\left(\v_n'(x_n)-\v_n'(x_0)\right)=0.
\end{equation}
Therefore, due to \eqref{4.14}, \eqref{4.15} and \eqref{4.16}, we find that, for sufficiently large $n$,
\begin{equation*}
-2s_n\int_{x_n}^{x_0}m'(x)\v_n'(x)\,dx<0,
\end{equation*}
which contradicts \eqref{4.12} and \eqref{4.13}. So, \eqref{4.6} and \eqref{4.8} cannot occur.
\par
Suppose \eqref{4.6} and \eqref{4.9}. Then, thanks to Theorem \ref{th2.1}(b), we find that
\begin{equation}
\label{4.17}
\v_n'(x)>0 \quad \hbox{for all} \;\; x \in (-1,1)
\end{equation}
for sufficiently large $n$. Moreover,
$$
  \lim_{n \rightarrow \infty}\int_{x_n}^{x_0}(\l_n-c(x))\v_n(x)\,dx=\int_{x_{\infty}}^{x_0}(\l_{\infty}-c(x))\v_{\infty}(x)\,dx
=\int_{x_{\infty}}^{x_0}(\l_{\infty}-c(x))\,dx>0.
$$
Therefore, for sufficiently large $n$,
\begin{equation*}
-2s_n\int_{x_n}^{x_0}m'(x)\v_n'(x)\,dx>0,
\end{equation*}
which contradicts \eqref{4.12} and \eqref{4.17}. So, \eqref{4.6} and \eqref{4.9} cannot occur neither.
\par
Suppose \eqref{4.7} and \eqref{4.8} hold. Then, for sufficiently large $n$, say $n\geq n_2$,
we have that
\begin{equation*}
x_0<x_n \quad \hbox{for all}\;\; n\geq n_2.
\end{equation*}
Moreover, by enlarging $n_2$, if necessary, we can assume that
\begin{equation}
\label{4.18}
c<\l_n \;\; \hbox{in} \;\; [-1,x_n) \;\; \hbox{and} \;\;
c>\l_n \;\; \hbox{in} \;\;  (x_n,1],
\end{equation}
and, due to (Hm), since $x_\infty>x_0$, we have that
\begin{equation}
\label{4.19}
m' \lneq 0 \;\; \hbox{in} \;\; [x_0,x_n]\;\; \hbox{for all}\;\; n\geq n_2.
\end{equation}
By Theorem \ref{th2.1}(a), it follows from \eqref{4.18} that
\begin{equation}
\label{4.20}
\v_n'(x)<0 \;\; \hbox{for all} \;\; x \in (-1,1).
\end{equation}
Now, integrating the differential equation of \eqref{1.1} in  $[x_0,x_n]$, we find that
\begin{equation}
\label{4.21}
\v_n'(x_0)-\v_n'(x_n) -2s_n\int_{x_0}^{x_n}m'(x)\v_n'(x)\,dx=\int_{x_0}^{x_n}(\l_n-c(x))\v_n(x)\,dx
\end{equation}
for all $n \geq n_2$. Thus, letting $n \rightarrow \infty$ in the right hand side of \eqref{4.21}, Theorem \ref{th3.1}, \eqref{4.8} and \eqref{4.10} imply that
\begin{equation}
\label{4.22}
\lim_{n \rightarrow \infty}\int_{x_0}^{x_n}(\l_n-c(x))\v_n(x)\,dx=\int_{x_0}^{x_{\infty}}(\l_{\infty}-c(x))\v_{\infty}(x)\,dx
=\int_{x_0}^{x_{\infty}}(\l_{\infty}-c(x))\,dx>0.
\end{equation}
On the other hand, by Theorem \ref{th3.1},
\begin{equation}
\label{4.23}
\lim_{n \rightarrow \infty}\left(\v_n'(x_0)-\v_n'(x_n) \right)=0.
\end{equation}
Consequently, by  \eqref{4.21}, \eqref{4.22} and \eqref{4.23}, it is apparent that, for sufficiently large $n\geq n_2$,
\begin{equation*}
-2s_n\int_{x_0}^{x_n}m'(x)\v_n'(x)\,dx>0,
\end{equation*}
which contradicts \eqref{4.19} and \eqref{4.20}. Thus, \eqref{4.7} and \eqref{4.8} cannot occur.
\par
Finally, suppose \eqref{4.7} ($x_\infty>x_0$) and \eqref{4.9}. Then, by Theorem \ref{th2.1}(b), for sufficiently large $n$,
\begin{equation}
\label{4.24}
\v_n'(x)>0 \quad \hbox{for all} \;\; x \in (-1,1).
\end{equation}
Moreover,
$$
  \lim_{n \rightarrow \infty}\int_{x_0}^{x_n}(\l_n-c(x))\v_n(x)\,dx=\int_{x_0}^{x_{\infty}}(\l_{\infty}-c(x))\v_{\infty}(x)\,dx
=\int_{x_0}^{x_{\infty}}(\l_{\infty}-c(x))\,dx<0.
$$
Therefore, it follows from \eqref{4.21} that, for sufficiently large $n$,
\begin{equation*}
-2s_n\int_{x_0}^{x_n}m'(x)\v_n'(x)\,dx<0,
\end{equation*}
which contradicts \eqref{4.19} and \eqref{4.24}. This contradiction ends the proof of the theorem in Case 1.
\par
\vspace{0.2cm}

\noindent\emph{Case 2.} Suppose that $c(x) = \l_{\infty}$ possesses two transversal roots $x_{\infty,1}$,
$x_{\infty,2} \in (-1,1)$ such that
 \begin{equation*}
x_{\infty,1}<x_{\infty,2}\;\; \hbox{and}\;\; \l_\infty=c(x_{\infty,1})=c(x_{\infty,2})\neq c(x_0).
\end{equation*}
As in Case 1, by the continuity of $c(x)$ and the transversality of the roots $x_{\infty,1}$
and $x_{\infty,2}$, it follows from \eqref{4.3} that there are two sequences, $\{x_{n,1}\}_{n\geq 1}$ and $\{x_{n,2}\}_{n\geq 1}$, such that
\begin{equation}
\label{4.25}
c(x_{n,i})=\l_n \;\; \hbox{and}\;\; \lim_{n \rightarrow \infty}x_{n,i}=x_{\infty,i},\quad i=1,2,
\end{equation}
where $x_{n,1}$ and $x_{n,2}$ are the first and the last transversal roots of the equation $c(x)=\l_n$ in $(-1,1)$, respectively. Next, we will differentiate between several cases according to the relative
positions of $x_{\infty,1}$, $x_{\infty,2}$ and $x_0$. Suppose that
\begin{equation}
\label{4.26}
-1 < x_{\infty,1} < x_{\infty,2}<x_0.
\end{equation}
Thanks to \eqref{4.25} and \eqref{4.26}, there exists $n_1 \in \mathbb{N}$ such that
\begin{equation*}
x_{n,1} <x_{n,2} <x_0 \;\; \hbox{for all}\;\; n\geq n_1.
\end{equation*}
Thus, by assumption (Hm),
\begin{equation}
\label{4.27}
m'\gneq 0 \;\; \hbox{in}\; \; [x_{n,2},x_0]\;\; \hbox{for all}\;\; n\geq n_1.
\end{equation}
Subsequently, we will assume that $c(x)$ is increasing if $x<y_0$ and decreasing if $x>x_0$; the proof
when $c(x)$ is decreasing if $x<y_0$ and increasing if $x>x_0$ is  omitted by repetitive.
Then,  by Theorem \ref{th2.1}(c)(ii), we find that
\begin{equation}
\label{4.28}
\v_n'(x)>0 \;\; \hbox{for all} \;\; x \in [x_{n,2},x_0].
\end{equation}
Moreover, according to \eqref{4.26}, it becomes apparent that
\begin{equation}
\label{4.29}
c<\l_{\infty} \;\; \hbox{in} \;\; [x_{\infty,2},x_0].
\end{equation}
Integrating the differential equation of \eqref{1.1} in $[x_{n,2},x_0]$, we find that
\begin{equation}
\label{4.30}
\v_n'(x_{n,2})-\v_n'(x_0)-2s_n\int_{x_{n,2}}^{x_0}m'(x)\v_n'(x)\,dx=\int_{x_{n,2}}^{x_0}(\l_n-c(x))\v_n(x)\,dx
\end{equation}
for all $n \geq n_1$. Thus, letting $n \rightarrow \infty$ in the right hand side of \eqref{4.30},  it follows from \eqref{4.25}, Theorem \ref{th3.1} and \eqref{4.29} that
\begin{equation}
\label{4.31}
\lim_{n \rightarrow \infty}\int_{x_{n,2}}^{x_0}(\l_n-c(x))\v_n(x)\,dx=\int_{x_{\infty,2}}^{x_0}(\l_{\infty}-c(x))\v_{\infty}(x)\,dx
=\int_{x_{\infty,2}}^{x_0}(\l_{\infty}-c(x))\,dx>0.
\end{equation}
According to  Theorem \ref{th3.1},
\begin{equation}
\label{4.32}
\lim_{n \rightarrow \infty}\left(\v_n'(x_{n,2})-\v_n'(x_0) \right)=0.
\end{equation}
Thus, due to \eqref{4.30}, \eqref{4.31} and \eqref{4.32}, we find that, for sufficiently large $n$, \begin{equation}
\label{s5}
-2s_n\int_{x_{n,2}}^{x_0}m'(x)\v_n'(x)\,dx>0,
\end{equation}
which contradicts \eqref{4.27} and \eqref{4.28}. Therefore, \eqref{4.26} cannot occur. Similarly, in case
\begin{equation}
x_0<x_{\infty,1} < x_{\infty,2}< 1,
\end{equation}
there exists  $n_2 \in \mathbb{N}$ such that
\begin{equation}
x_0<x_{n,1}<x_{n,2} \;\; \hbox{for all} \;\; n\geq n_2.
\end{equation}
Thus, by (Hm),
\begin{equation}
\label{4.33}
m' \lneq 0 \;\; \hbox{in} \;\; [x_0,x_{n,1}] \;\; \hbox{for all} \;\; n\geq n_2.
\end{equation}
Moreover, since $c<\l_n$ in $[x_0,x_{n,1}]$, because, by construction,
\begin{equation}
\label{4.34}
c<\l_{\infty} \quad \hbox{in} \;\; [x_0,x_{\infty,1}],
\end{equation}
it follows from Theorem \ref{th2.1}(c)(i) that
\begin{equation}
\label{4.35}
\v_n'(x)<0 \;\; \hbox{for all} \;\; x \in [x_0,x_{n,1}].
\end{equation}
As in the previous case, integrating the differential equation of \eqref{1.1} in  $[x_0,x_{n,1}]$, we
find that
\begin{equation}
\label{4.36}
\v_n'(x_0)-\v_n'(x_{n,1})-2s_n\int_{x_0}^{x_{n,1}}m'(x)\v_n'(x)\,dx=\int_{x_0}^{x_{n,1}}(\l_n-c(x))\v_n(x)\,dx
\end{equation}
for all $n \geq n_2$. Consequently, letting $n \rightarrow \infty$ in the right hand side of \eqref{4.36}, it follows from \eqref{4.25},
Theorem \ref{th3.1},  and \eqref{4.34} that
\begin{equation}
\label{4.37}
\lim_{n \rightarrow \infty}\int_{x_0}^{x_{n,1}}(\l_n-c(x))\v_n(x)\,dx=\int_{x_0}^{x_{\infty,1}}(\l_{\infty}-c(x))\v_{\infty}(x)\,dx
=\int_{x_0}^{x_{\infty,1}}(\l_{\infty}-c(x))\,dx>0.
\end{equation}
On the other hand, due to Theorem \ref{th3.1},
\begin{equation}
\label{4.38}
\lim_{n \rightarrow \infty}\left(\v_n'(x_0)-\v_n'(x_{n,1}) \right)=0.
\end{equation}
Thus, \eqref{4.36}, \eqref{4.37} and \eqref{4.38} imply that, for sufficiently large $n$,
\begin{equation*}
-2s_n\int_{x_0}^{x_{n,1}}m'(x)\v_n'(x)\,dx>0,
\end{equation*}
which contradicts \eqref{4.33} and \eqref{4.35}. Therefore, necessarily,
$$
  -1 < x_{\infty,1} <x_0< x_{\infty,2} < 1.
$$
Then, by \eqref{4.25}, there exists $n_3 \in \mathbb{N}$ such that
\begin{equation}
\label{u1}
x_{n,1}<x_0<x_{n,2} \;\; \hbox{for all}\;\; n\geq n_3.
\end{equation}
In this case, according to (Hm), we have that
\begin{equation}
\label{4.39}
m' \gneq 0 \;\; \hbox{in} \;\;  [x_{n,1},x_0],  \;\; \hbox{and} \;\;
m'\lneq 0 \;\; \hbox{in} \;\; [x_0,x_{n,2}].
\end{equation}
Also, as we are assuming that $c$ satisfies (Hc)(a), we have that
\begin{equation}
\label{4.40}
\l_{\infty}<c \;\; \hbox{in} \;\; [x_{\infty,1},x_{\infty,2}],\;\;\hbox{and}\;\; c<\l_\infty \;\;
\hbox{in}\;\; [-1,x_{\infty,1})\cup (x_{\infty,2},1].
\end{equation}
According to \eqref{4.40},
$$
   -\v_n''(-1)=(\l_n-c(-1))\v_n(-1)>0, \quad -\v_n''(1)=(\l_n-c(1))\v_n(1)>0.
$$
Thus, $\v_n''(-1)<0$ and $\v_n''(1)<0$ for sufficiently large $n$. Therefore, $\v_n$ has, at least, a critical point in $(-1,1)$; at least, a minimum. According to Theorem \ref{th2.1}(c)(i,ii), the critical points of
$\v_n$ must lie in $[x_{n,1},x_{n,2}]$. Thus, by Theorem \ref{th2.1}(c)(iii), we can conclude that  $\v_n$ possesses a unique critical point in $(-1,1)$ at some $z_n \in (x_{n,1},x_{n,2})$. Moreover, $z_n$ is a minimum. By the Bolzano--Weierstrass theorem, there exists a subsequence of $z_n$, again labeled by $n$, such that
\begin{equation}
\label{lz}
\lim_{n \rightarrow \infty}z_n=z_{\infty}\in[x_{\infty,1},x_{\infty,2}].
\end{equation}
Then, either $z_{\infty}<x_0$, or $z_{\infty}>x_0$, or $z_{\infty}=x_0$. Suppose that
$z_{\infty}<x_0$. Then, for sufficiently large $n$,
\begin{equation}
\label{4.41}
x_{n,1} <z_n<x_0.
\end{equation}
Since $[x_0,x_{\infty,2}] \subset [x_{\infty,1},x_{\infty,2}]$, \eqref{4.40} implies that
\begin{equation}
\label{4.42}
\l_{\infty}<c \;\; \hbox{in} \;\;  [x_0,x_{\infty,2}].
\end{equation}
Also, since $[x_0,x_{n,2}] \subset [z_n,x_{n,2}]$, by \eqref{4.41}, Theorem \ref{th2.1}(c)(iii) guarantees that
\begin{equation}
\label{4.43}
\v_n' \gneq 0 \;\; \hbox{in} \;\;  [x_0,x_{n,2}].
\end{equation}
Moreover, integrating the differential equation of \eqref{1.1} in $[x_0,x_{n,2}]$ yields to
\begin{equation}
\label{4.44}
\v_n'(x_0)-\v_n'(x_{n,2}) -2s_n\int_{x_0}^{x_{n,2}}m'(x)\v_n'(x)\,dx=\int_{x_0}^{x_{n,2}}(\l_n-c(x))\v_n(x)\,dx.
\end{equation}
As in all previous cases, letting $n \rightarrow \infty$ in the right hand side of \eqref{4.44},
and using \eqref{4.25} and Theorem \ref{th3.1}, it follows from \eqref{4.42} that
\begin{equation}
\label{4.45}
\lim_{n \rightarrow \infty}\int_{x_0}^{x_{n,2}}(\l_n-c(x))\v_n(x)\,dx=\int_{x_0}^{x_{\infty,2}}(\l_{\infty}-c(x))\v_{\infty}(x)\,dx
=\int_{x_0}^{x_{\infty,2}}(\l_{\infty}-c(x))\,dx<0.
\end{equation}
Similarly, by Theorem \ref{th3.1},
\begin{equation}
\label{4.46}
\lim_{n \rightarrow \infty}\left(\v_n'(x_0)-\v_n'(x_{n,2}) \right)=0.
\end{equation}
Therefore, according to \eqref{4.44}, \eqref{4.45} and \eqref{4.46}, we find that, for sufficiently large $n$, \begin{equation}
\label{4.47}
-2s_n\int_{x_0}^{x_{n,2}}m'(x)\v_n'(x)\,dx<0.
\end{equation}
As \eqref{4.47} contradicts \eqref{4.39} and \eqref{4.43}, necessarily $z_\infty\geq x_0$. Suppose that  $z_\infty > x_0$.  Then, working in the interval $[x_{n,1},x_0]$, instead of in $[x_0,x_{n,2}]$, a similar contradiction can be attaint as in the case when $z_\infty < x_0$. Therefore,
$z_\infty=x_0$.  In this case, there exists a subsequence of $z_n$, relabeled by $z_n$, such that either $z_n$ converges to $z_{\infty}=x_0$ from its right, or from its left. Assume that it converges to $z_{\infty}=x_0$ from its right, i.e. $z_n\geq z_\infty=x_0$ for sufficiently large $n$, and consider the interval
$$
   [x_{n,1},x_0]=[x_{n,1},z_\infty] \subset [x_{n,1},z_n].
$$
Since $[x_{\infty,1},x_0] \subsetneq [x_{\infty,1},x_{\infty,2}]$, by \eqref{4.40}, we have that
\begin{equation}
\label{4.48}
\l_{\infty}<c \;\; \hbox{in} \;\; [x_{\infty,1},x_0].
\end{equation}
Moreover, by construction,
\begin{equation}
\label{4.49}
\v_n'\lneq 0 \;\; \hbox{in} \;\; [x_{n,1},x_0]=[x_{n,1},z_\infty] \subset [x_{n,1},z_n].
\end{equation}
Now, integrating the differential equation of \eqref{1.1} in  $[x_{n,1},x_0]$ yields
\begin{equation}
\label{4.50}
\v_n'(x_{n,1})-\v_n'(x_0)-2s_n\int_{x_{n,1}}^{x_0}m'(x)\v_n'(x)\,dx=
\int_{x_{n,1}}^{x_0}(\l_n-c(x))\v_n(x)\,dx.
\end{equation}
Letting $n \rightarrow \infty$ in the right term of \eqref{4.50} it follows from \eqref{4.25}, Theorem \ref{th3.1} and \eqref{4.48} that
\begin{equation}
\label{4.51}
\lim_{n \rightarrow \infty}\int_{x_{n,1}}^{x_0}(\l_n-c(x))\v_n(x)\,dx=
\int_{x_{\infty,1}}^{x_0}(\l_{\infty}-c(x))\v_{\infty}(x)\,dx=\int_{x_{\infty,1}}^{x_0}(\l_{\infty}-c(x))\,dx<0.
\end{equation}
As usual,  due to Theorem \ref{th3.1},
\begin{equation}
\label{4.52}
\lim_{n \rightarrow \infty}\left(\v_n'(x_{n,1})-\v_n'(x_0) \right)=0.
\end{equation}
Consequently, according to \eqref{4.50}, \eqref{4.51} and \eqref{4.52}, we find that, for sufficiently large $n$,
\begin{equation*}
-2s_n\int_{x_{n,1}}^{x_0}m'(x)\v_n'(x)\,dx<0,
\end{equation*}
which contradicts \eqref{4.39} and \eqref{4.49}. This contradiction shows that $z_n$ must converge to
$z_\infty=x_0$ from its left, i.e. $z_n \leq x_0$ for sufficiently large $n$. In such case, we can make
the previous integration in the interval $[x_0,x_{n,2}]$, instead of in $[x_{n,1},x_0]$, to reach a
similar contradiction. Therefore, as we have got a contradiction in all previous cases, the proof
is completed in Case 2.
\par
\vspace{0.2cm}
\noindent\emph{Case 3.} Assume that $c(x)=\lambda_{\infty}$ has a unique root $x_{\infty} \in (-1,1)$ which is not transversal. Then, since we are assuming that $c(x)$ satisfies (Hc)(a), it is apparent that $x_\infty=y_0$ and, owing to \eqref{4.5},
$$
  \l_\infty=c(y_0)\neq c(x_0).
$$
Thus,  $y_0\neq x_0$. Now, suppose that, in addition, $c(x)$ has a maximum at $y_0$. Then, by Lemma \ref{le2.1},  $\lambda_n < \lambda_{\infty}$ for all $n \geq 1$ and $\lambda_n=c(x)$ has two transversal roots, $x_{n,1}$ and $x_{n,2}$, such that
$$
\lim_{n \rightarrow \infty} x_{n,1}=\lim_{n \rightarrow \infty} x_{n,2} =x_\infty=y_0,
$$
and
$$
  x_{n,1}<y_0=x_{\infty}<x_{n,2}\;\;\hbox{and}\;\; c(x_{n,1})=c(x_{n,2})=\lambda_n \;\; \hbox{for all}\;\; n \geq 1.
$$
Since $x_0\neq y_0$, the proof can be completed adapting the argument of Case 2. By repetitive, we are
omitting the technical details here. Similarly, one can get a contradiction when $y_0$ is a minimum of $c(x)$.
\par
\vspace{0.2cm}
\noindent\emph{Case 4.} Suppose that $c(x)=\lambda_{\infty}$ has no roots in $(-1,1)$ and that $y_0$ is a maximum of $c(x)$. In this case, by Lemma \ref{le2.1}, some of the following options occurs. Either
\begin{equation}
\label{4.53}
\lambda_{\infty}=c(-1) \;\; \hbox{and} \;\; c(x) > \lambda_{\infty} \;\; \hbox{for all} \;\; x \in (-1,1],
\end{equation}
or
\begin{equation}
\label{4.54}
\lambda_{\infty}=c(1) \;\; \hbox{and} \;\; c(x) > \lambda_{\infty} \;\; \hbox{for all} \;\; x \in [-1,1),
\end{equation}
or
\begin{equation}
\label{4.55}
\lambda_{\infty}=c(-1)=c(1) \;\; \hbox{and} \;\; c(x) > \lambda_{\infty} \;\; \hbox{for all} \;\; x \in (-1,1).
\end{equation}
Assume that \eqref{4.53} holds. Then, $c_L=c(-1)$ and, since $\lim_{n \rightarrow \infty}\lambda_n=\lambda_{\infty}$,  it follows from Lemma \ref{le2.1} that $\lambda_n>\lambda_{\infty}$ for all $n \geq 1$. Moreover, since $c(1)>c(-1)=c_L$, the equation $c(x)=\l_n$ has a unique transversal root,
$x_{n}\in(-1,1)$, for sufficiently large $n$. By construction,
$$
\lim_{n \rightarrow \infty}x_n=-1 \;\; \hbox{and} \;\; c(x)>\lambda_n \;\; \hbox{for all} \;\;
x\in (x_n,1].
$$
In particular, for  sufficiently large $n$, we have that
$$
-1<x_n<x_0 \;\; \hbox{and} \;\; c(x)>\lambda_n \;\; \hbox{for all} \;\; x \in (x_n,x_0].
$$
Thus, working in $[x_n,x_0]$ with the arguments of Case 1, we can reach a contradiction.
This ends the proof under the assumption \eqref{4.53}. Similarly, one can get a contradiction when
\eqref{4.54} holds by integrating the differential equation of \eqref{1.1} in the intervals
$[x_0,x_n]$ for sufficiently large $n$. By repetitive, we are omitting the technical details here.
\par
Finally, assume that \eqref{4.55} holds. Then, since
$$
   \lim_{n \rightarrow \infty}\lambda_n=\lambda_{\infty}\;\; \hbox{and}\;\;  c_L=c(-1)=c(1),
$$
it follows from Lemma \ref{le2.1} that $\lambda_n>\lambda_{\infty}$ for all $n \geq 1$. Moreover, there are two sequences, $x_{n,1}$ and $x_{n,2}$, $n\geq 1$, with $-1<x_{n,1}<x_{n,2}<1$, where $x_{n,1}$ and $x_{n,2}$ are the unique roots of the equation $c(x)=\lambda_n$, which are transversal roots, such that
$$
\lim_{n \rightarrow \infty}x_{n,1} =-1, \;\;
\lim_{n \rightarrow \infty}x_{n,2} =1, \;\; \hbox{and} \;\;
c(x)>\lambda_n \;\; \hbox{for all} \;\; x\in (x_{n,1},x_{n,2}).
$$
In particular,
$$
-1<x_{n,1} <x_0<x_{n,2}<1 \;\; \hbox{for sufficiently large}\;\; n\geq 1.
$$
Adapting the argument already given in the third situation analyzed in Case 2, a contradiction can be
as well reached. A similar contradiction can be reached when $y_0$ is a minimum of $c(x)$, though we omit
the details here by repetitive. This completes the proof of the theorem in Case 4.
\par
When $c(x)$ is monotone, the equations $c(x)=\l_n$ always have a unique transversal root. Thus, in
such case,  the proof can be easily completed adapting the arguments of Case 1 under assumption (Hc)(a)
if $x_\infty\in (-1,1)$.  When $x_\infty =\pm 1$, one should adapt the arguments already given in Case 4 under assumption (Hc)(a) to deliver the complete details of the proof, that we omit here by repetitive.
This concludes the proof.
\end{proof}

\section{Numerical Simulations}

\noindent The pseudospectral method used in this paper has been  previously used to continue branches of solutions for a non  local problem in \cite{HLGM}, for nodal solutions  in \cite{LGMR},  for metacoexistence states of a  Lotka--Volterra diffusive competition system with refuges in \cite{FM3}, as well as to calculate metasolutions in the disk in \cite{FM1}.  In \cite{FM2}, there have been calculated the differentiation matrices  in an interval and in  the disk,  subject to  nonhomogeneous  Dirichlet, Neumann and Robin boundary conditions, as well as for the Biharmonic Equation subject to nonhomogeneous Dirichlet boundary conditions.
\par
Eigenvalue boundary value problems are crucial  in  innumerable applications, as,  for example, in the detection of bifurcation points in nonlinear equations and systems. Problem \eqref{1.1} is a key model to solve a wide variety of nonlinear problems. Concerning the relation between eigenvalue problems, Chebyshev Methods and applications,  one of the most important examples  is the Orr--Sommerfeld equation,  an eigenvalue problem for a fourth order differential equation that has been used to test  the  efficiency of  spectral methods  (see \cite{CMAT} pp. 193--200). A  pioneering work to solve this problem was \cite{Orzag}, where the critical Reynolds number was calculated numerically using a Tau Chebyshev method.
\par
Pseudospectral methods are very appreciated because of its unrivalled  convergence properties, either as \emph{spectral accuracy}, also called \emph{infinite-order convergence} when the solution is an  infinitely differentiable function, or \emph{exponential convergence}, in the case that the solution is an analytic function (see Chapters 7 and 8 of \cite{Tre3}, or  page 47 of \cite{CMAT}). Particularly, for solutions  having finite regularity,   global error estimates   can be found  in  Chapter 7 of \cite{CMAT} or Section 5.1.2 of \cite{CMAT}.
\par
The pseudospectral method  used in this paper consists of two steps:
\begin{enumerate}
\item[$\ast$] Firstly, to approximate the first eigenfunction $ {\varphi_{s}}$  by its
interpolation polynomial ${\varphi_{s,N}}(x)$  of degree $N$ based on the collocation points of Chebyshev--Gauss--Lobatto (CGL),
$$
    x_i  =  \cos \left( \dfrac{(i-1) \> \pi}{N}\right), \quad i=1,...,N+1.
$$

\item[$\ast$] Secondly, to obtain an algebraic eigenvalue problem by forcing  $ \left( {\lambda_{s,N}}, {\varphi_{s,N}}(x)\right)$ to satisfy the linear equation in \eqref{1.1} at the collocation points
$x_i$, $i=2,...,N$,  and the boundary conditions on $x_1$ and $x_{N+1}$.
\end{enumerate}
Then, the  eigenvalue and eigenvector
$$
    \left( {\lambda_{s,N}}, {\varphi_{s,N}}(x_i)_{i=1}^{i=N+1} \right)
$$
are calculated applying the MATLAB command \emph{eigs} to the differentiation matrix obtained applying the pseudospectral method to \eqref{1.1}. The implementation of   the boundary conditions in the differentiation matrix follows \cite{FM2}.
\par
In this section, the results of  the simulations are discussed for eight different eigenvalue boundary value problems:  A1, A2, A3,  B1, B2, B3, B2A and B2B, whose coefficients $m(x)$ and $c(x)$ are listed in Table \ref{appsol} and $c(x)$ are shown in Figure \ref{FigM1}, with
$$
  c_1(x):=40(2-\cos(x-1/3)), \qquad x\in [-1,1].
$$
In each one of these cases,  for  $N=801$, Table \ref{appsol} collects the values of $c_M-c_L$, $s$, $\|{\varphi_{s,N}}-1 \|_\infty$, $\| {\varphi^{\prime}_{s,N}} \|_\infty,$ ${ \lambda_{s,N}}-c(0)$ and the condition  number, $ \kappa({\lambda_{s,N} })$, of computing the eigenvalue $\lambda_{s,N}$. Because of the presence of the  advection term $-2s m^{\prime}(x)$ in \eqref{1.1}, it is needed a further specific study to make sure  that the problem is well conditioned (see Section 4.3.3 of  \cite{CMAT}, or \cite[pp. 83--88]{P}).  For this reason we are including in Table \ref{appsol} the condition numbers $\kappa({\lambda_{s,N} })$, which has been calculated from the definition in page 474  of \cite{Tre2}.
\par

\begin{table} [h] 
\begin{center}
\begin{tabular}{|l|l|l|c|c|c|c|c|c|c|c|c|}
\hline

\multicolumn{9}{|c|}{$ {\varphi_s}(x) \approx  {\varphi_{s,N}}$, ${\lambda_{s} }\approx{\lambda_{s,N}}$, N=801} \\
\hline
Ex. & $m(x)$ &  $c(x)$ & \!$c_M\!-\!c_L$\! &  $s$  & $\!\|{\varphi_{s,N}}\!-\!1 \|_\infty\!$ &  $\| {\varphi^{\prime}_{s,N}}\|_\infty$ & \! ${\lambda_{s,N}}\!-\!c(0)$ \! & $\kappa({\lambda_{s,N} })$ \\
\hline
A1 & $1-x^4$  & $c_1(x)$  & 40.59  &  $10^6$ & 3.0 e-04 & 4.0 e-03 & 4.5 e-03 &  7.4 \\
\hline
A2 & $1-x^{10}$  &  $c_1(x)$ & 40.59  &  $10^6$ &1.1  e-01 &  3.2 e-01 & 2.6 e-01 & 2.55  \\
\hline
A3 & $1-x^{16}$  &  $c_1(x)$ & 40.59  &  $10^6$ & 4.0 e-01 & 7.6 e-01 & 4.7 e-01 & 1.94 \\
\hline
B1 & $1-x^4$  &  $2+x$  & 2& $10^4$  &  6.5 e-04 & 3.2 e-03 & -6.1 e-06  & 4.16 \\
\hline
B2 & $1-x^{10}$  &  $2+x$  & 2 &  $10^4$ &3.2 e-02 & 6.3 e-02 & -2.2 e-03 & 2.02 \\
\hline
B3 & $1-x^{16}$  &  $2+x$  & 2 &  $10^4$ & 9.3 e-02& 1.3 e-01 & -1.0 e-02 & 1.66 \\
\hline
B2A &  $1-x^4$  &  $2+0.1 x$  &  0.2 & $10^4$ & 6.5 e-05& 3.2 e-04 & -6.1 e-08 & 4.16 \\
\hline
B2B& $1-x^{4}$  &  $2+10 x$  & 20 &  $10^4$ & 6.5 e-03&3.2 e-02 & -6.1 e-04 & 4.16 \\
\hline
\end{tabular}
\caption{ {\small  Eight cases of numerical study, with $c_1(x):=40(2-\cos(x-1/3))$ }}
\label{appsol}
\end{center}
\end{table}

Figures  \ref{FigM1} and  \ref{FigM2} show the plots of the coefficients $c(x)$  and ${ \lambda_{s,N}}-c(0)$, respectively. The first row of Figure \ref{FigM3} shows a zoom of the advection terms $-2s m^{\prime}(x)$ for values of $s\in [10^{-2},10^6]$  for  $m(x)=1-x^4$  (left), $m(x)=1-x^{10}$ (centre),  and  $m(x)=1-x^{16}$ (right). The second row shows the corresponding zooms for values of $s\in [10^{-2},10^4]$ for  $m(x)=1-x^4$  (left),  $m(x)=1-x^{10}$ (centre),  and  $m(x)=1-x^{16}$ (right). Finally, Figures \ref{FigM4}, \ref{FigM5} and \ref{FigM6} show the plots of some of the computed eigenfunctions for each of the cases collected in Table
\ref{appsol}, together with the plots of their respective derivatives.

\begin{figure}[h!]
\begin{center}
\begin{tabular}{lll}
\includegraphics[width=0.31\columnwidth] {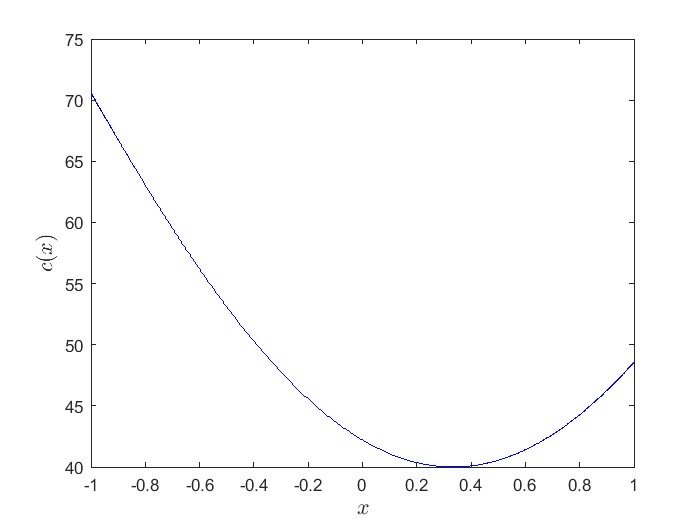}
& \includegraphics*[width=0.31\columnwidth]{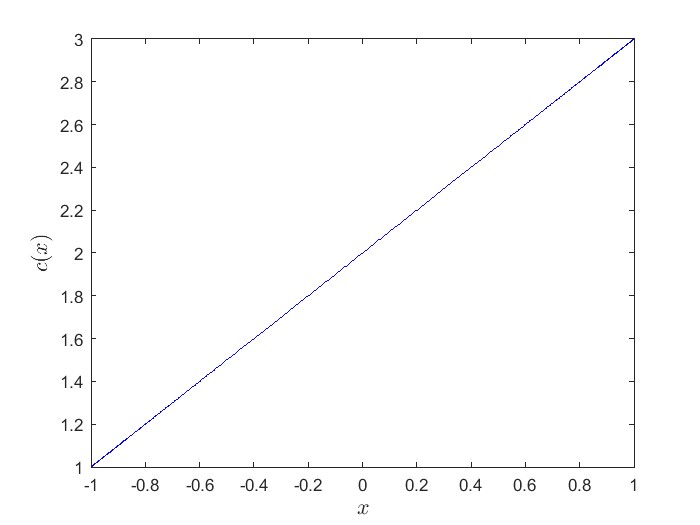}
& \includegraphics*[width=0.31\columnwidth]{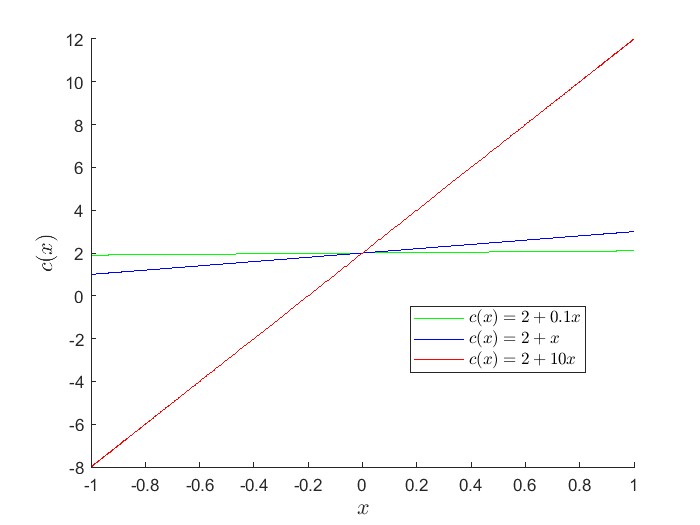}
\end{tabular}
\end{center}
\caption{  \small  Graphs  of  the coefficients $c_1(x)=40(2-\cos(x-1/3))$ (left),   $c_2(x)=2+x$ (centre), and $c_3(x)=2+m x$  (right) for the slopes  $m=0.1, m= 1$ and $m=10$.}
\label{FigM1}
\end{figure}

\begin{figure}[h!]
\begin{center}
\begin{tabular}{lll}
\includegraphics[width=0.31\columnwidth] {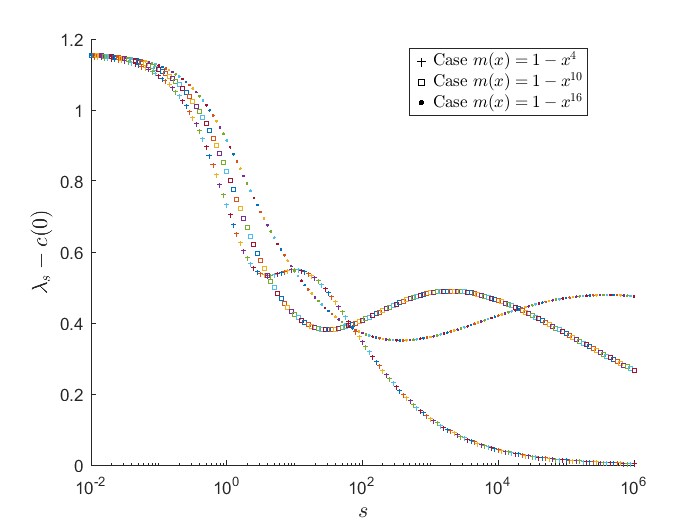}
& \includegraphics*[width=0.31\columnwidth]{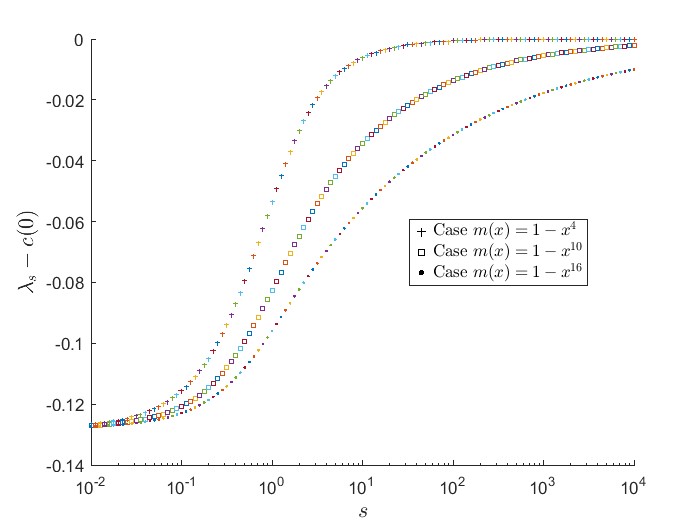}
& \includegraphics*[width=0.31\columnwidth]{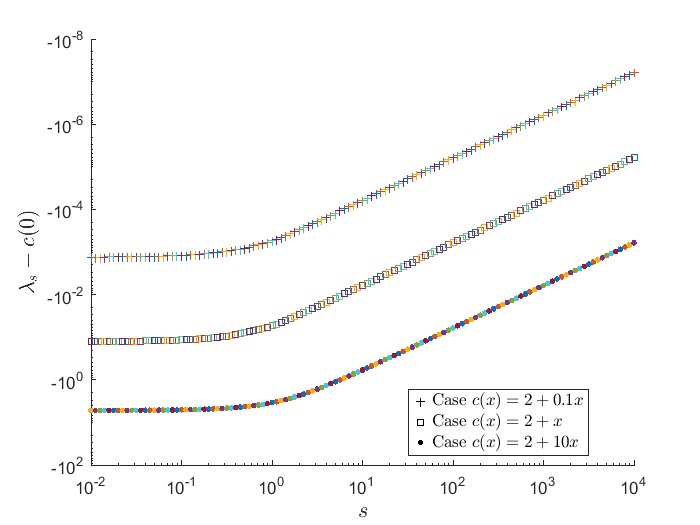}
\\
\end{tabular}
\end{center}
\caption{  \small  Graphs  of $\l_{s,N}-c(0)$, with $\l_s=\l(s)$,  for examples A1, A2 and A3 (left);  for examples B1,B2 and B3 (centre), and for examples B1, B2A and B2B (right).}
\label{FigM2}
\end{figure}

At the light shared by our numerical study, some conclusions come over:
\begin{enumerate}
 \item[a)] Fixed $ c(x)$,  the condition number  $\kappa({\lambda_{s,N} })$ decreases as the power of $x$ in $m(x)$, $n=2\nu$, increases. This occurs because $m'(x)$ achieves smaller values and consequently the influence of the advection term is smaller.
 \vspace{0.2cm}
 \par
 \item[b)] Similarly, fixed $c(x)$,   the convergence of $\lambda_{s,N}$ to $c(0)$  slows down as $n=2\nu$ increases. This occurs because a bigger $s$ is needed  to catch the information of the advection   term $-2s m'(x)$  at zero, as illustrated by Figure \ref{FigM3}.
 \vspace{0.2cm}
 \par
 \item[c)] Fixed  $n=2\nu$, the convergence of $\lambda_{s,N}$ to $c(0)$  slows down  as  the variation $c_M-c_L$ of $c(x)$ in $[-1,1]$ increases. This may be caused by the integral term depending of $c(x)$ in the Rayleigh Quotient.
 \vspace{0.2cm}
 \par
 \item[d)] Taking into account the condition numbers $\kappa({\lambda_{s,N} })$ listed in the last column of Table \ref{appsol}, it is concluded that the numerical simulations of this paper are robust,  and that  rounding errors do not affect the numerical results of this paper.
 \vspace{0.2cm}
 \par
 \item[e)] In the cases A1, A2 and A3, the highest value of $s$ is  $s=10^6$, while in cases B1, B2, B3, B2A and B2B the highest $s$ is  $ s=10^4$. Allowing $s$ to be bigger than these  upper bounds,
     makes appear instabilities in  the calculation of  ${\lambda_{s,N}}$  that may push us to increment the value of $N$, but in this paper the number of collocation points is fixed at $N+1=802$ in all simulations.
\end{enumerate}

\begin{figure}[h!]
\begin{center}
\begin{tabular}{lll}
\includegraphics[width=0.21\columnwidth] {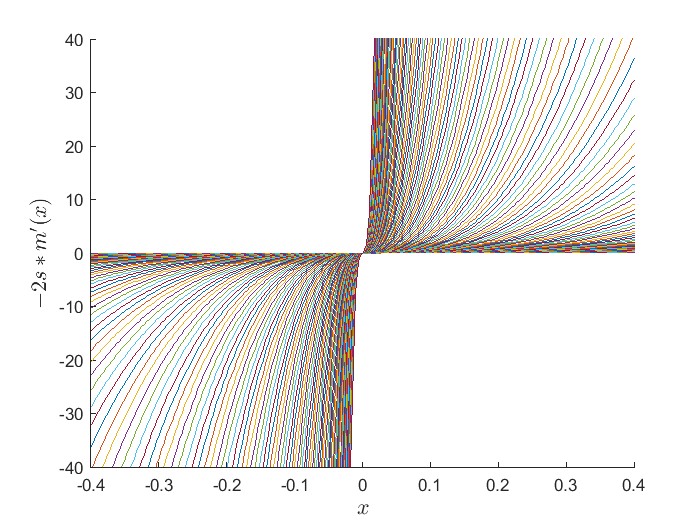}
& \includegraphics*[width=0.21\columnwidth]{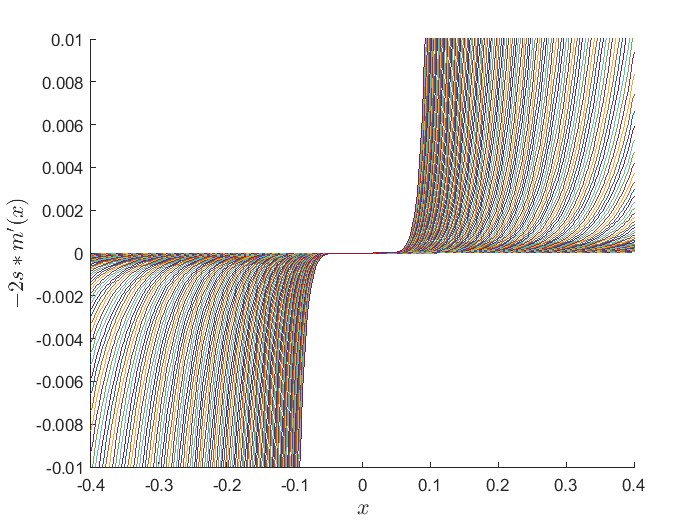}
& \includegraphics*[width=0.21\columnwidth]{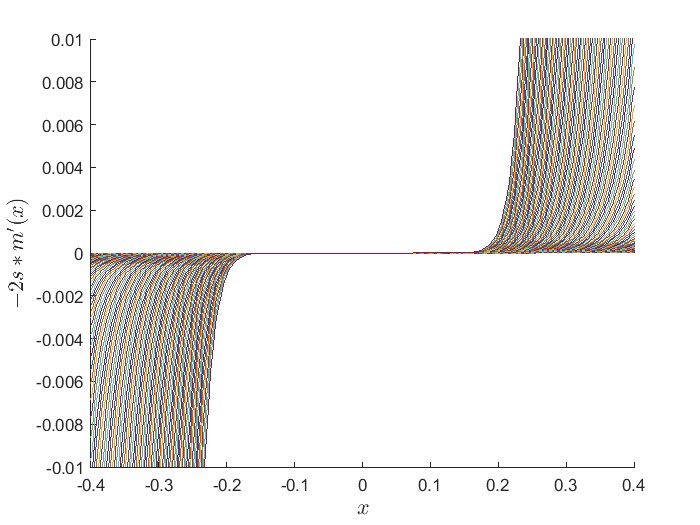}
\\
\includegraphics[width=0.21\columnwidth] {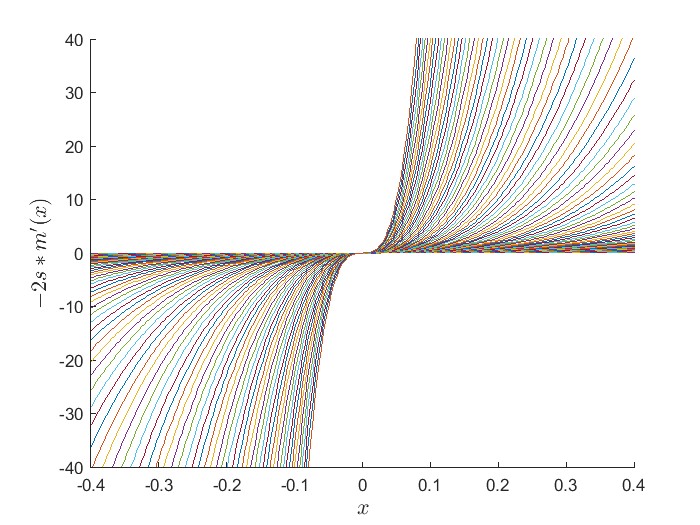}
& \includegraphics*[width=0.21\columnwidth]{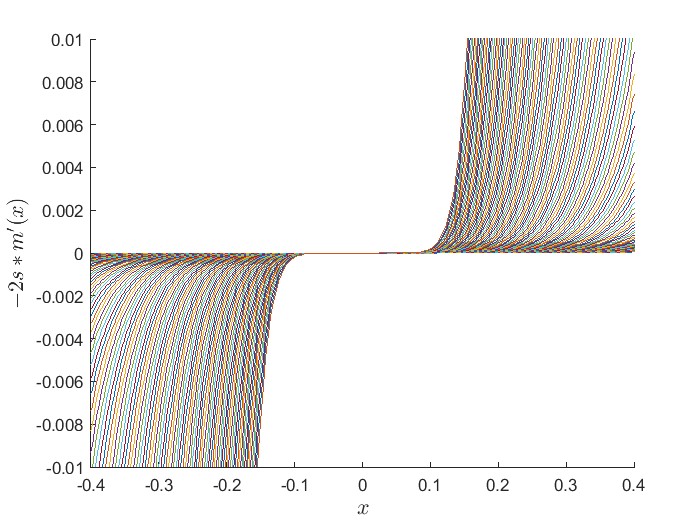}
& \includegraphics*[width=0.21\columnwidth]{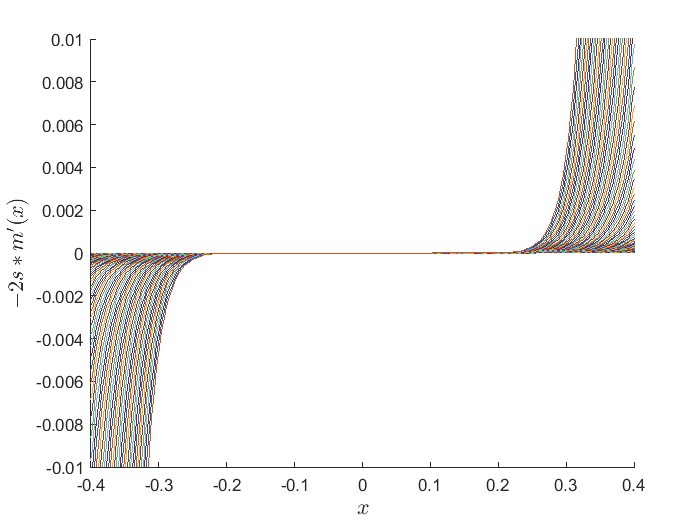}
\end{tabular}
\end{center}
\caption{ { \small Zoom of the  advection  terms $-2s m^{\prime}(x)$, for values of $s\in [10^{-2},10^6]$  for  $m(x)=1-x^4$  (left), 	$m(x)=1-x^{10}$ (centre),  and  $m(x)=1-x^{16}$ (right) (top) and zoom of the  advection  terms $-2s m^{\prime}(x)$, for values of $s\in [10^{-2},10^4]$  for  $m(x)=1-x^4$  (left),  $m(x)=1-x^{10}$ (centre),  and  $m(x)=1-x^{16}$ (right) (bottom).}}
\label{FigM3}
\end{figure}

\begin{figure}[h!]
\begin{center}
\begin{tabular}{lll}
\includegraphics[width=0.21\columnwidth] {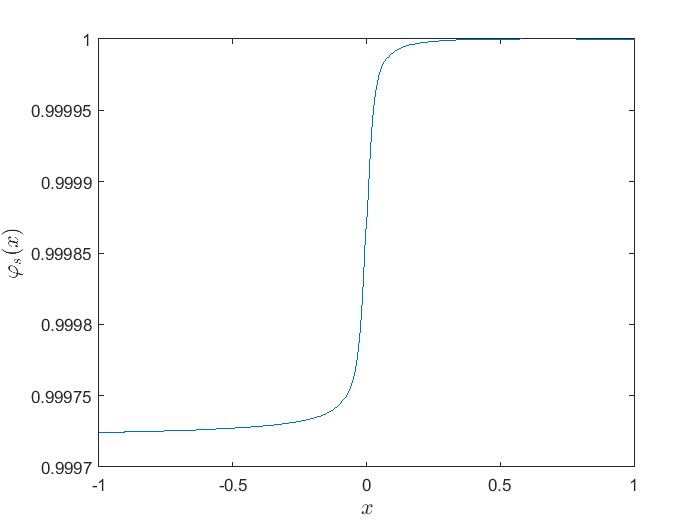}
& \includegraphics*[width=0.21\columnwidth]{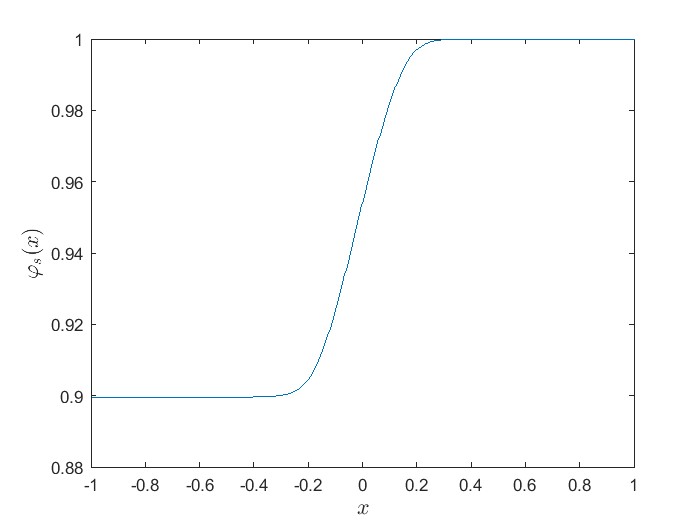}
& \includegraphics*[width=0.21\columnwidth]{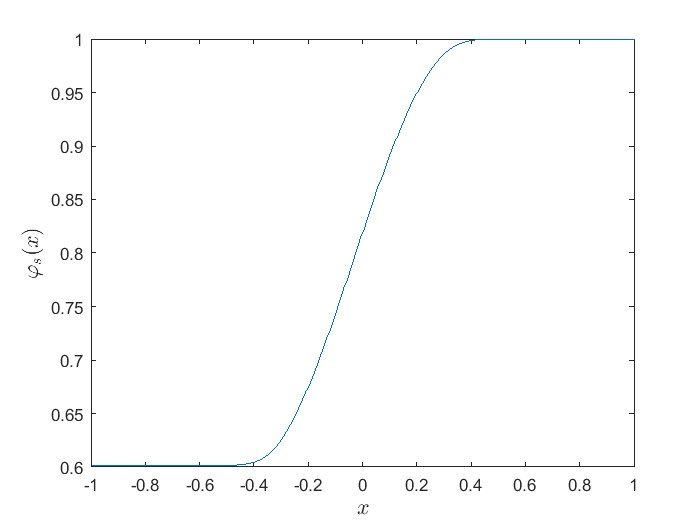}
\\
\includegraphics[width=0.21\columnwidth] {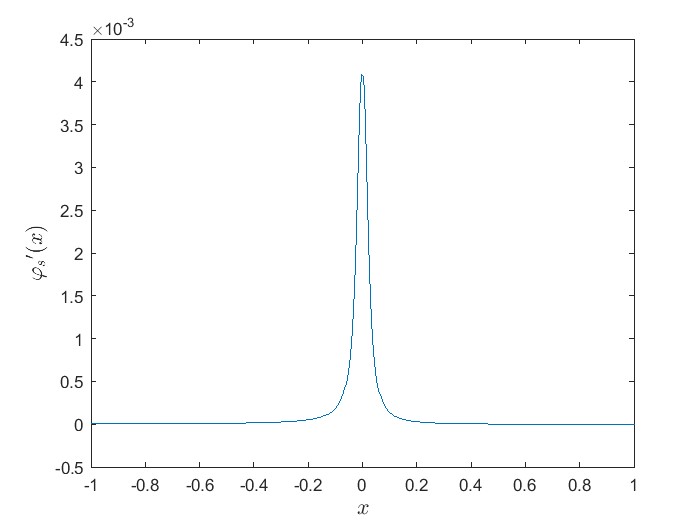}
& \includegraphics*[width=0.21\columnwidth]{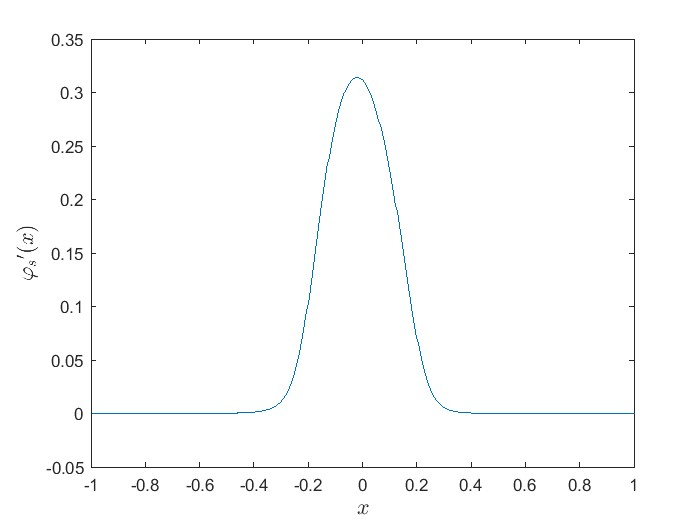}
& \includegraphics*[width=0.21\columnwidth]{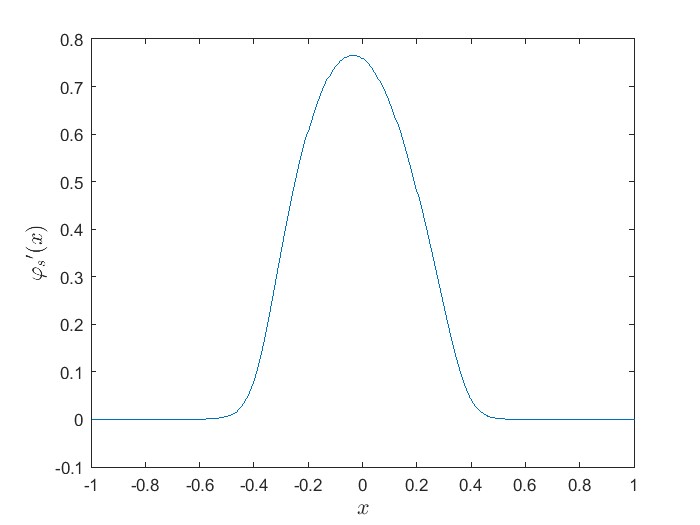}
\end{tabular}
\end{center}
\caption{ { \small Examples A1, A2 and A3: Plots of the normalized eigenfunctions (top) and its derivatives (bottom).}}
\label{FigM4}
\end{figure}

\begin{figure}[h!]
\begin{center}
\begin{tabular}{lll}
\includegraphics[width=0.21\columnwidth] {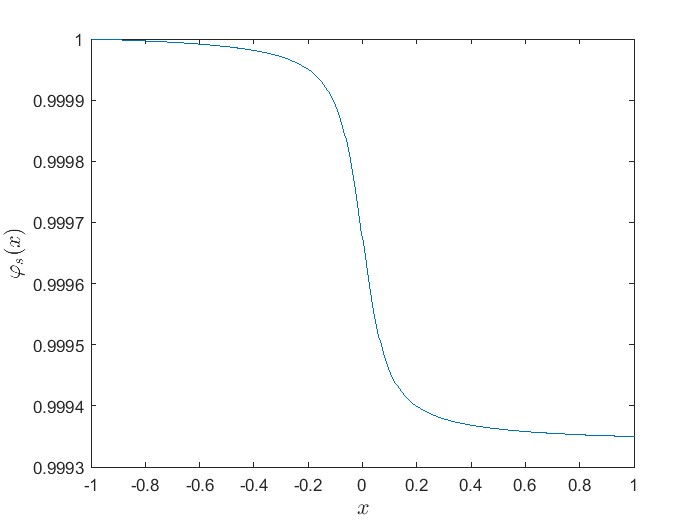}
& \includegraphics*[width=0.21\columnwidth]{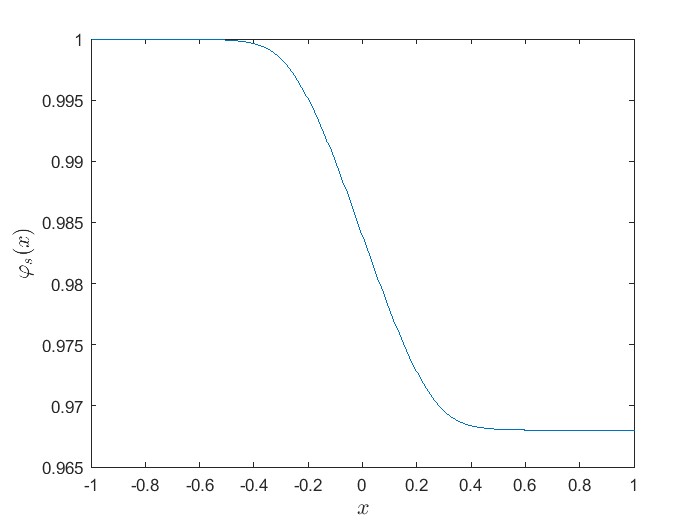}
& \includegraphics*[width=0.21\columnwidth]{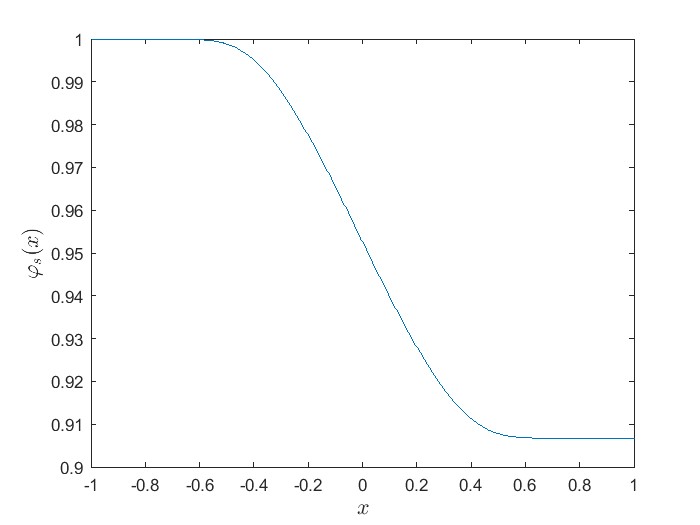}
\\
\includegraphics[width=0.21\columnwidth] {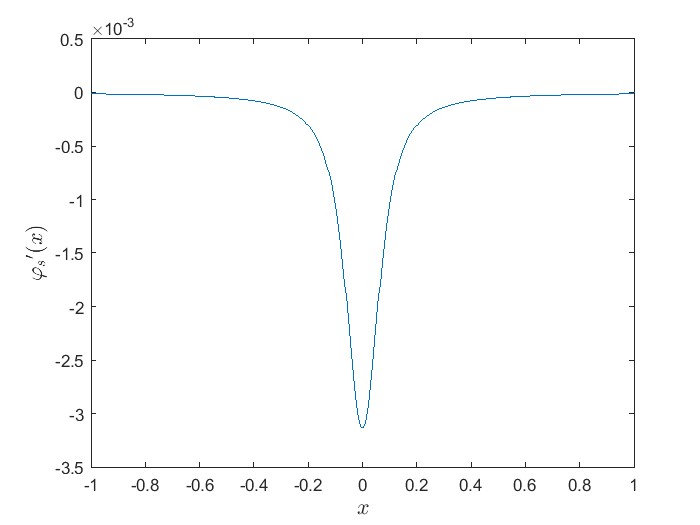}
& \includegraphics*[width=0.21\columnwidth]{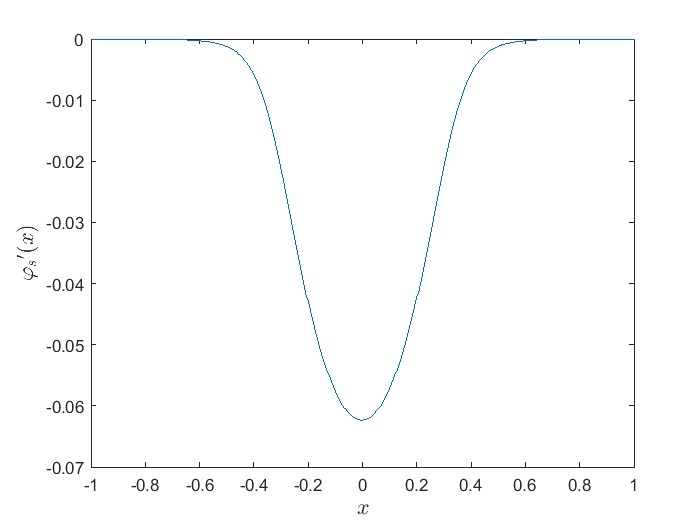}
& \includegraphics*[width=0.21\columnwidth]{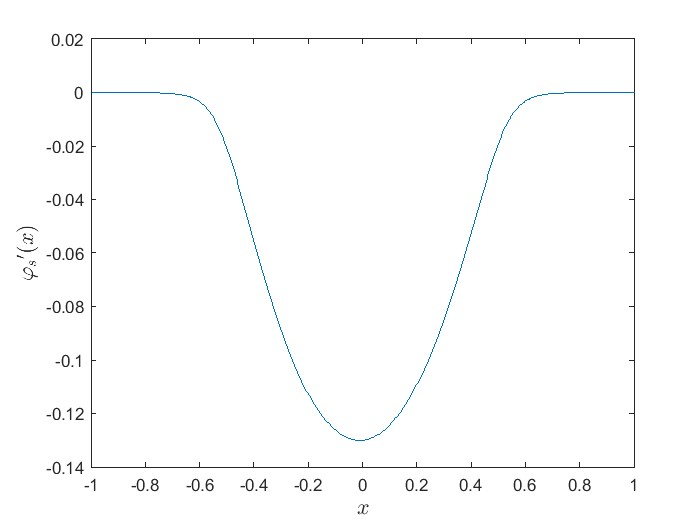}
\end{tabular}
\end{center}
\caption{ { \small Examples B1, B2 and B3: Plots of the normalized eigenfunctions (top) and its derivatives (bottom).}}
\label{FigM5}
\end{figure}

\begin{figure}[h!]
\begin{center}
\begin{tabular}{lll}
\includegraphics[width=0.21\columnwidth] {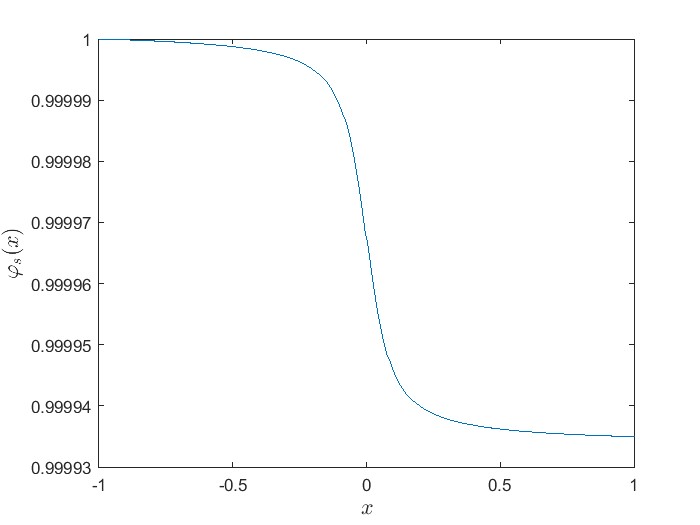}
& \includegraphics*[width=0.21\columnwidth]{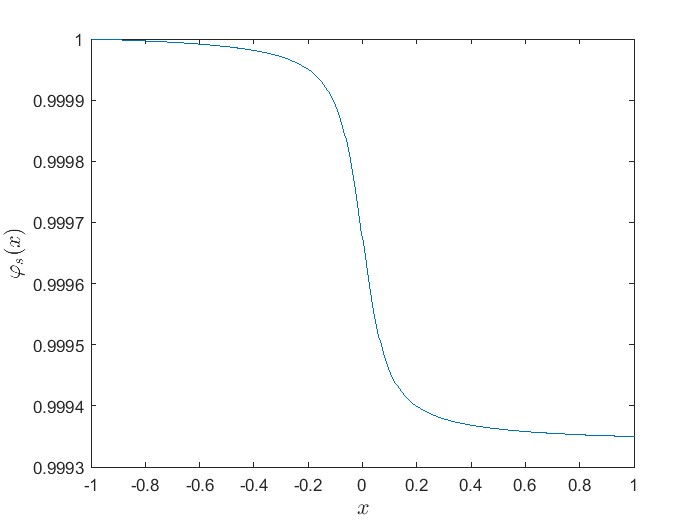}
& \includegraphics*[width=0.21\columnwidth]{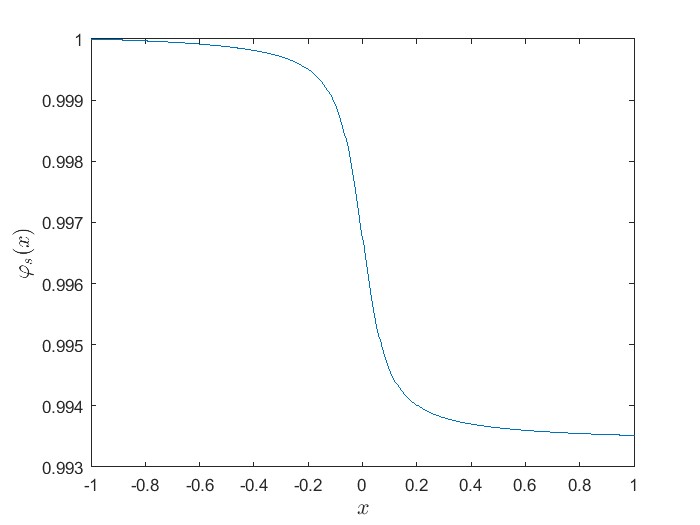}
\\
\includegraphics[width=0.21\columnwidth] {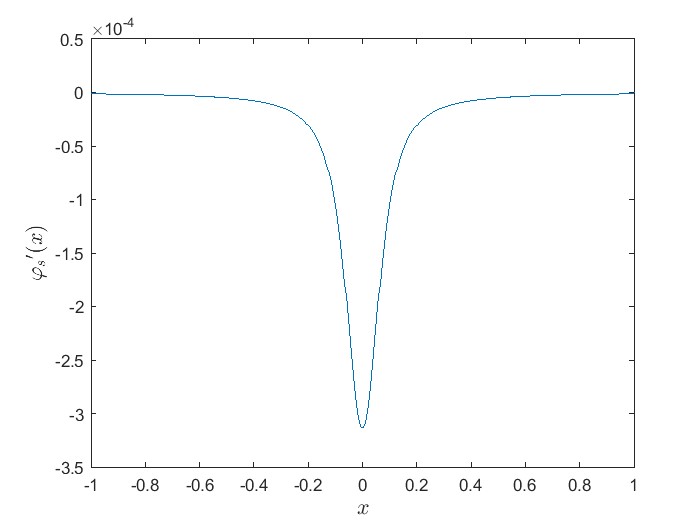}
& \includegraphics*[width=0.21\columnwidth]{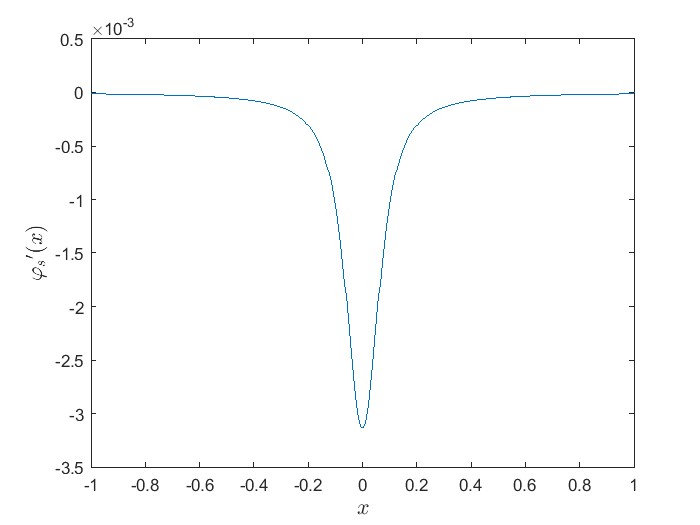}
& \includegraphics*[width=0.21\columnwidth]{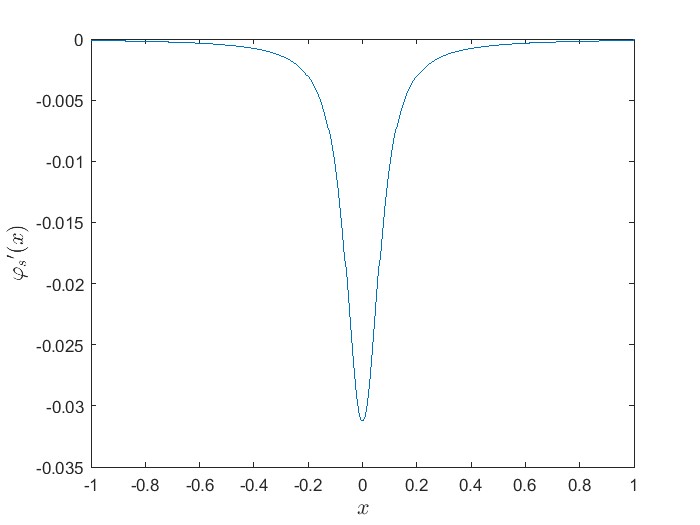}
\end{tabular}
\end{center}
\caption{ { \small Examples B2A, B1 and B2B: Plots of the normalized eigenfunctions (top) and its derivatives (bottom).}}
\label{FigM6}
\end{figure}

\end{document}